\documentclass[fullpage]{amsart}
\usepackage{pictex}
\usepackage{amsmath}
\usepackage{amssymb}
\usepackage{amsopn}
\usepackage{epsfig}
\usepackage{amsfonts}
\usepackage{latexsym}
\usepackage{todonotes}
\usepackage{hyperref}
\usepackage[utf8]{inputenc}
\usepackage[margin=3.0cm]{geometry}

\newtheorem{theorem}{Theorem}[section]
\newtheorem{lemma}[theorem]{Lemma}
\newtheorem{proposition}[theorem]{Proposition}

\newtheorem{remark}{Remark}[section]
\newtheorem{example}{Example}
\newtheorem{algorithm}{Algorithm}

\newcommand{\R}{\mathbb R} \newcommand{\Z}{\mathbb Z}
\newcommand{\N}{\mathbb{N}} \newcommand{\Q}{\mathbb{Q}}

\title{A ``strange'' continued fraction associated with the Romik map}
\author{Yufei Chen}
\address{Jiangsu University, School of Mathematical Sciences, Xuefu 301, Zhenjiang, 212013, P.R.CHINA} 
\email{yufechen@163.com}
\author{Karma Dajani}
\address{Department of Mathematics, Utrecht University, P.O. Box 80010, 3508 TA Utrecht, the Netherlands}
\email{k.dajani@uu.nl}
\author{Yanyan Hu}
\address{School of Mathematics and Physics, University of Science and Technology Beijing, Beijing 100083, China} 
\email{yanyanhu1993@gmail.com}
\author{Cor Kraaikamp}
\address{Delft University of Technology, EWI (DIAM), Mekelweg 4, 2628 CD Delft, the Netherlands} 
\email{c.kraaikamp@tudelft.nl}

\begin{document}

\begin{abstract}
In~2008, Dan Romik studied in this journal \emph{Primitive Pythagorean Triples}, or PPTs. In order to do so, he introduced a modified slow (subtractive) Euclidean algorithm, and showed that the underlying dynamical system of this Euclidean algorithm (the ``Romik system''), is ergodic and has a $\sigma$-finite, infinite measure, of which is explicitly given.  

In this paper, the Romik system is further studied. Various basic properties are determined, such as the expansion of rational numbers and quadratic irrationals. Also (a version of) the planar natural extension of the Romik system is obtained, and the $\sigma$-finite, invariant measure is explicitly given, and it is shown that it is ergodic. Furthermore, for Lebesgue almost every $x$ asymptotically half of the regular continued fraction (RCF) convergents of $x$ are among the Romik convergents. We also show that related to the Romik map a ``strange'' continued fraction can be given. ``Strange,'' as the set of possible partial quotients (i.e., \emph{digits}) for any $x\in [0,1]$ in this expansion is $\{ 0, \pm 2\}$. Various properties of this ``Romik expansion'' are given.
\end{abstract}

\maketitle

\section{Introduction}\label{sec:Introduction}
Let $x\in\R$ have as \emph{regular continued fraction} (RCF) expansion $x=[a_0;a_1,a_2,\dots]$, where $a_0\in\Z$, such that $x-a_0\in [0,1)$, and\footnote{Recall that the RCF-expansion of $x\in\R$ is finite if and only if $x\in\Q$. In that case $x$ has two RCF-expansions. In all other cases the RCF-expansion of $x\in\R$ is unique.}  $a_n\in\N$, for $n\geq 1$. Furthermore, by finite truncation, we get the sequence of \emph{regular continued fraction convergents} $(p_n/q_n)$ of $x$, given by
\begin{equation}\label{RCFconvergents}
\frac{P_n}{Q_n} = a_0 + \frac{1}{a_1+ \frac{\displaystyle 1}{\displaystyle a_2+\ddots +\frac{1}{a_n}}},
\end{equation}
where $P_n,Q_n\in\Z$, $Q_n>0$, and $\text{gcd}\{ P_n,Q_n\}=1$, for $n\geq 1$. We denote~(\ref{RCFconvergents}) by $P_n/Q_n=[a_0;a_1,\dots,a_n]$. For more information of the RCF and its metrical and arithmetical properties, see~\cite{[DK1], [IK], [Kh], [RS]}.\smallskip\

The following result by Legendre from 1798 (\cite{[L],[BJ]}) gives an important motivation why continued fractions should be studied. Let $A,B\in\Z$ with $B>0$, and $\text{gcd}\{ A,B\}=1$ be such, that 
\begin{equation}\label{legendre}
B^2\left| x-\frac{A}{B}\right| < \frac{1}{2},
\end{equation}
then the rational number $A/B$ is an RCF-convergent of $x$. So there exists an $n\in\N$, such that $A=P_n$ and $B=Q_n$. Legendre's result states that \emph{if} we want to approximate an irrational number $x$ well by a rational number $A/B$ (such that~\eqref{legendre} holds), this rational number is an RCF-convergent of $x$. In fact, in~\cite{[BJ]} it is shown, that if
$$
B^2\left| x-\frac{A}{B}\right| < 1,
$$
we have that $A/B$ is either an RCF-convergent, or a so-called first or last mediant convergent. For $1\leq k\leq a_{n+1}-1$, mediant convergents (or: mediants) are of the form
$$
\frac{kP_n+P_{n-1}}{kQ_n+Q_{n-1}}.
$$
By definition, there is no mediant convergent if $a_{n+1}=1$. The sequence of RCF-convergents and mediants are ``generated'' by the \emph{Farey tent map} $F$. This map and its natural extension were introduced and studied by Shunji Ito in 1989 (\cite{[I]}); see~\cite{[D], [P]} and also~\cite{[DKS], [S]}. The Farey tent map $F:[0,1]\to [0,1]$, which Ito called the \emph{mediant convergent transformation}, is defined by
$$
F(x) = \begin{cases}
\displaystyle{\frac{x}{1-x}}, & \text{for $0\leq x\leq \frac{1}{2}$}\medskip\ ;\\
\displaystyle{\frac{1-x}{x}}, & \text{for $\tfrac{1}{2}\leq x\leq 1$}.
\end{cases}
$$
Setting $X=[0,1]$, on p.~564 of~\cite{[I]}, Ito mentions that the following theorem was proved in~\cite{[D], [P]}.
\begin{theorem}\label{them:DandP}
The mediant convergent transformation $(X,F)$ has $\sigma$-finite invariant measure $\mu$
$$
d\mu = \frac{dx}{x},
$$
and the dynamical system $(X, F, \mu )$ is ergodic.
\end{theorem}

Note that there is a close relation between the Farey tent map and the Gauss (or: {\it regular continued fraction}) map $G:[0,1)\to [0,1)$, defined by
\begin{equation}\label{GaussMap}
G(x) = \frac{1}{x} - \left\lfloor \frac{1}{x}\right\rfloor ,\quad \text{for $x\in (0,1)$, and $G(0)=0$}.
\end{equation}
The relation between $F$ and $G$ is given by (see also~(1,12) in~\cite{[I]})
\begin{equation}\label{jumpFtoR}
G(x) = F^k(x),\quad \text{if $x\in \Big[ \tfrac{1}{k+1},\tfrac{1}{k}\Big)$, for $k\in\N$}.
\end{equation}
So $G$ is an acceleration of $F$, also known as a {\it jump transformation}; see also~\cite{[Sch]}. As a consequence, the regular continued fraction convergents of any $x\in\R$ form a subsequence of the Farey-convergents of $x$. Originally, for the Farey tent map $F$ there was no continued fraction expansion related to it. However, in~\cite{[DK2]} it was shown that the so-called Lehner map $L: [1,2)\to [1,2)$, defined for $x\in [1,2)$ by
$$
L(x) = \begin{cases}
\displaystyle{\frac{1}{2-x}}, & \text{if $1\leq x<\tfrac{3}{2}$};\medskip\ \\
\displaystyle{\frac{1}{x-1}}, & \text{if $\tfrac{3}{2}\leq x <2$}
\end{cases}
$$
yields all the regular and mediant convergents of $x$, and also that using the map $L$ repeatedly we find a continued fraction expansion of $x$ of the form
$$
x = b_0+\frac{e_1}{b_1 +\displaystyle{\frac{e_2}{b_2+\ddots }}} = [b_0;e_1/b_1,e_2/b_2,\dots],
$$
which was originally found by Joe Lehner in~\cite{[Leh]}. Here, for $i\geq 0$,
$$
(b_i,e_{i+1}) = \begin{cases}
(1,1), & \text{if $L^i (x)\in \big[ \tfrac{3}{2},2\big)$}; \\
(2,-1), & \text{if $L^i (x)\in \big[1,\tfrac{3}{2}\big)$}.    
\end{cases}
$$
In general, for $c_n\in\{\pm 1\}$, $d_0\in\Z$ and $d_n\in\N$, for $n\in\N$, the continued fraction expansion
$$
d_0+\frac{c_1}{d_1 +\displaystyle{\frac{c_2}{d_2+\ddots }}}
$$
will be denoted\footnote{With the obvious restriction on $n$ when the continued fraction expansion is finite.} by $[d_0;c_1/d_1,c_2/d_2,\dots]$. When all $c_n$ are equal to $+1$ (which is the case of the RCF), we simply delete them.\smallskip\

It is shown in~\cite{[DKS]} that the Lehner map $L$ is  conjugate to the Farey tent map $F$ through the translation $S:x\mapsto x+1$; it is easy to see that for $x\in [1,2]$
$$
L(x) = S\circ F\circ S^{-1}(x).
$$
Defining $\epsilon: [0,1]\to \{ 0, 1\}$ by
$$
\epsilon (x) = \begin{cases}
0, & \text{if $x\leq \tfrac{1}{2}$}; \\
1, & \text{if $x> \tfrac{1}{2}$},
\end{cases}
$$
and setting for $x\in [0,1]$ and $n\geq 0$, $x_n=L^n(x)$, $\epsilon_{n+1}=\epsilon_{n+1}(x)= \epsilon (x_n)$, and if $[b_0;e_1/b_1,e_2/b_2,\dots]$ is the Lehner expansion of $x+1$, then the \emph{Farey continued fraction expansion} of $x$ is defined as
$$
x = [b_0-1;e_1/b_1,e_2/b_2,\dots],
$$
where
$$
(b_{i+1},e_i) = (2-\epsilon_{i+1}, 2\epsilon_{i+1}-1).
$$
See also Section~1.4 in~\cite{[S]} and Section 4 in~\cite{[DKS]}. Since Ito found in~\cite{[I]} the invariant measure and ergodicity for the Farey tent map $F$, using the conjugacy $S:x\mapsto x+1$, the invariant measure and ergodicity of the Lehner map $L$ follow immediately; see also~\cite{[DK1]}.

\section{The Romik map}\label{sec:RomikMap}
In order to study Pythagorean triples, Dan Romik introduced in~\cite{[R]} a variant of the Farey tent map, which Dong Han Kim, Seul Bee Lee and Lingmin Liao called in~\cite{[KLL]} the \emph{Romik map} $R:[0,1]\to [0,1]$, defined by:
\begin{equation}\label{RomikMapR}
R(x) = \begin{cases}
\displaystyle{\frac{x}{1-2x}}, & \text{for $0\leq x\leq \frac{1}{3}$}\medskip\ ;\\
\displaystyle{\frac{1}{x}-2}, & \text{for $\tfrac{1}{3}\leq x\leq\tfrac{1}{2}$}\medskip\ ;\\
\displaystyle{2-\frac{1}{x}}, & \text{for $\tfrac{1}{2}\leq x\leq 1$},
\end{cases}
\end{equation}
see Figure \ref{fig:Romikmap}.
\begin{figure}[h]
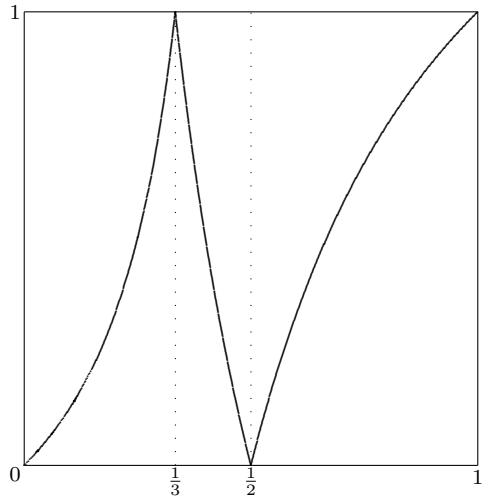

$$
\beginpicture
    \setcoordinatesystem units <0.6cm,0.6cm>
    \setplotarea x from -1 to 11, y from 0 to 11
    \put {\footnotesize $0$} at -0.2 -0.2
    \put {\footnotesize $1$} at 10 -0.25
    \put {\footnotesize $1$} at -0.2 10
    \put {\footnotesize $\tfrac{1}{3}$} at 3.3333 -0.35
    \put {\footnotesize $\tfrac{1}{2}$} at 5 -0.35

\setquadratic
 \plot 
 0 0
 0.25 0.26316
 0.3 0.319148
 0.35 0.376344
 0.4 0.43478
 0.5 0.555555
 0.6 0.68181818
 0.7 0.81395348
 0.9 1.09756097
 1  1.25
 1.1 1.410256
 1.2 1.578947
 1.5 2.142857
 2 3.333333
 2.5 5
 2.75 6.111111
 2.9 6.9047619
 3 7.5
 3.1 8.15789473
 3.15 8.5135135 
 3.2 8.88888 
 3.25 9.28571428
 3.3 9.70588235 
 3.31 9.7928994
 3.32 9.88095238
 3.33 9.97005988
 3.33333 10.0001
 /

\setquadratic
 \plot
 3.33333 10.0001
 3.4 9.4117647
 3.5 8.57142857
 3.6 7.77777
 3.7 7.02702702
 3.8 6.31578947
 3.9 5.64102564
 4 5
 4.1 4.3902439
 4.2 3.8095238
 4.3 3.25581395
 4.4 2.72727272
 4.5 2.22222
 4.6 1.73913043
 4.7 1.27659574
 4.8 0.833333
 4.9 0.40816326
 4.95 0.202020202
 5 0
 /

\setquadratic
 \plot
 5 0
 5.1 0.39215686
 5.2 0.76923076
 5.3 1.13207547
 5.4 1.48148148
 5.5 1.818181818
 5.6 2.14285714
 5.7 2.45614035
 5.8 2.75862069
 5.9 3.05084745
 6 3.3333333
 6.1 3.60655737
 6.2 3.87096774
 6.3 4.12698412
 6.4 4.375
 6.5 4.61538461
 6.6 4.84848484848
 6.7 5.07462686
 6.8 5.29411764
 6.9 5.50724637
 7 5.71428571
 7.1 5.91549295
 7.2 6.1111111
 7.3 6.30136986
 7.4 6.486486486
 7.5 6.666666
 7.6 6.84210526
 7.7 7.01298701
 7.8 7.17948717
 7.9 7.34177215
 8 7.5
 8.1 7.65432098
 8.2 7.80487804
 8.3 7.95180722
 8.4 8.09523809
 8.5 8.23529411
 8.6 8.37209302
 8.7 8.50574712
 8.8 8.63636363
 8.9 8.76404494
 9 8.8888888
 9.1 9.01098901
 9.2 9.13043477
 9.3 9.24731182
 9.4 9.36170212
 9.5 9.47368421
 9.6 9.5833333
 9.7 9.69072164
 9.8 9.79591836
 9.9 9.8989898
 10 10
/

\putrule from 0 10 to 10 10
\putrule from 10 0 to 10 10
\putrule from 0 0 to 10 0
\putrule from 0 0 to 0 10

\setdots
\putrule from 3.33333 0 to 3.33333 10
\putrule from 5 0 to 5 10
\endpicture
$$ 
\caption{The Romik map $R$}\label{fig:Romikmap}
\end{figure}

Romik showed that $R$ is ergodic, and has a $\sigma$-finite, infinite invariant measure $\mu$ absolutely continuous with respect to Lebesgue measure and defined by 
\begin{equation}\label{density}
\mu(A)=\int_A\frac{1}{\sqrt{2}}\cdot \frac{1}{x(1-x)}{\rm d}\lambda(x),
\end{equation}
where $\lambda$ is Lebesgue measure and $A$ is a Borel subset of $[0,1]$. Since $\mu$ is an infinite measure, the factor $\frac{1}{\sqrt 2}$ is superfluous, so we omit it from the rest of the paper.

\smallskip\

There is a close relation between the Farey tent map $F$ and the Romik map $R$
\begin{equation}\label{relationRomikFarey}
R(x) = \begin{cases}
F^2(x), & \text{if $x\in [0,\tfrac{1}{2})$}; \\
1-F(x), & \text{if $x\in [\tfrac{1}{2},1]$}.
\end{cases}
\end{equation}
We see that $R$ is either an accelerated version of $F$ (on $[0,\tfrac{1}{2}]$), or a ``flipped'' version of the RCF-map $G$ (which is also the Farey tent map $F$) on $[\tfrac{1}{2},1]$; see also~\cite{[DHKM]} for more information on ``flipped expansions.'' Clearly, $R$ is simply the RCF-map $G$ from~\eqref{GaussMap} (the ``Gauss'' map) on $(\tfrac{1}{3},\tfrac{1}{2}]$.

Recently, Dong Han Kim, Seul Bee Lee, and Lingmin Liao used in~\cite{[KLL]} the Romik map to study the \emph{odd-odd continued fraction} (OOCF) expansion, which was obtained as a jump transformation of the Romik map. In a similar way they also obtained Fritz Schweiger's \emph{even} continued (ECF, in~\cite{[KLL]} called EICF for \emph{even integer} continued fraction) fraction expansion from the Romik-system (also as a jump transformation). In~\cite{[KL]} it was shown how the ECF-expansion of any $x\in\R$ can be derived from the RCF-expansion of $x$ via \emph{insertions} and \emph{singularizations}. Insertions and singularizations will be introduced and used in the proof of Proposition~\ref{prop:relationRomikGauss} below. In~\cite{[KLL]}, Theorem~5.3, Kim \emph{et al}.\ obtain a similar result for the relation between the RCF-expansion and the OOCF-expansion of any $x\in\R$. Singularizations and insertions also are used in a conversion algorithm in Subsection~\ref{subsec:conversion} to convert the RCF-expansion of any $x\in[0,1)$ to its Romik-expansion.\medskip\

In comparison to the Farey tent map, much less is known about the Romik map. A natural question is: how do these systems relate to each other? Is the Farey-system ``larger'' than the Romik system, or the other way around? I.e., for an irrational $x$, is the sequence of convergents of $x$ `given' by the Romik map a subsequence of the Farey tent map? In view of the first line in~\eqref{relationRomikFarey} one would expect that. But how about the second line in~\eqref{relationRomikFarey}? Are there Romik-convergents which are not Farey-convergents? Or is the Farey-expansion after all \textbf{not} ``the mother of all (continued fraction) expansions''? The question whether all RCF-convergents are Romik-convergents is settled in Example~\ref{example2} and Theorem~\ref{thm:densityRCFconvergents} below (they are not!).\medskip\

There are many more immediate questions. For example, for the Farey tent map originally there was not a continued fraction map related to it, but then it turned out that an isomorphic version of the Lehner-map does the trick; see~\cite{[DKS], [S]}. Is there something similar for the Romik map? Can we associate a continued fraction expansion with it? In this paper, we will answer this question in the affirmative. But first we investigate what the Romik map ``does'' with rational and quadratic irrational numbers $x$.

\subsection{Elementary properties of the Romik map}\label{subsec:elementaryproperties}
Note that for $n\in\N$, $n\geq 3$ we have that:
$$
R(\tfrac{1}{n}) = \frac{\frac{1}{n}}{1-\frac{2}{n}} = \frac{1}{n-2},
$$
so we see that for these values of $n$ we have:
\begin{equation}\label{stepsRomikmap1}
R\left( [\tfrac{1}{n+1},\tfrac{1}{n})\right) = \Big[ \frac{1}{n-1},\frac{1}{n-2}\Big).
\end{equation}
Furthermore, for $n\in\N$, $n\geq 1$:
$$
R(\tfrac{n}{n+1}) = 2 - \frac{1}{\frac{n}{n+1}} = 2 - \frac{n+1}{n} = \frac{n-1}{n}, 
$$
so we see that for $n\geq 1$ we have that:
\begin{equation}\label{Rimageinterval}
R\left( [\tfrac{n}{n+1},\tfrac{n+1}{n+2})\right) = \Big[ \frac{n-1}{n},\frac{n}{n+1}\Big).
\end{equation}
It immediately follows from definition~\eqref{sec:RomikMap} of the Romik map $R$ and Figure~\ref{fig:Romikmap} that
\begin{equation}\label{stepsRomikmap2}
R\left( [0,\tfrac{1}{3}]\right) = [0,1] = R\left( [\tfrac{1}{3},\tfrac{1}{2}]\right) = R\left( [\tfrac{1}{2},1]\right);
\end{equation}
i.e., the Romik map $R$ is \emph{full} on the intervals $\left[ 0,\tfrac{1}{3}\right]$, $\left[ \tfrac{1}{3},\tfrac{1}{2}\right]$, and $\left[ \tfrac{1}{2},1\right]$.\medskip\

In the next proposition, we show that the orbit of a rational number $\tfrac{a}{b}\in (0,1)$ under $R$ is finite.
\begin{proposition}\label{prop:RationalsHaveFiniteOrbit}
Let $a,b\in\N$, $a<b$, then there exists an $n\in\N$, such that $R^n(\tfrac{a}{b}) \in\{ 0,1\}$. 
\end{proposition}

\begin{proof}
For the proof we discern the following three cases.
\begin{itemize}
\item[($i$)] Let $\tfrac{a}{b}\in (0,\tfrac{1}{3}]$. Then $0<3a\leq b$, and therefore $b-2a\geq a>0$, and we have:
$$
0 < \frac{p}{q} = R(\tfrac{a}{b}) = \frac{\frac{a}{b}}{1-\frac{2a}{b}} = \frac{a}{b-2a}\leq \frac{a}{a}=1.
$$
So $q=b-2a$ (or $q$ is a divisor of $b-2a$ if $\tfrac{p}{q}=\frac{a}{b-2a}$ can be further simplified), and since $b-2a<b$ we see that $q<b$.\smallskip\

\item[($ii$)]  Let $\tfrac{a}{b}\in (\tfrac{1}{3},\tfrac{1}{2}]$. Then $2a\leq b$ and $b<3a$, and we have $0\leq b-2a<a$,
$$
0\leq \frac{p}{q} = R(\tfrac{a}{b}) = \frac{1}{\frac{a}{b}} - 2 = \frac{b}{a} - 2 = \frac{b-2a}{a} < \frac{a}{a}=1.
$$
So $\frac{p}{q} = \frac{b-2a}{a}$, and we see that $q\leq a < b$.\smallskip\

\item[($iii$)] Let $\tfrac{a}{b}\in (\tfrac{1}{2},1]$. Then $a\leq b$ and $b < 2a$, and we see that $0< 2a-b\leq a$, and therefore,
$$
\frac{p}{q} = R(\tfrac{a}{b}) = 2 - \frac{1}{\frac{a}{b}} = 2-\frac{b}{a}  = \frac{2a-b}{a}\leq \frac{a}{a}=1.
$$
So again we find that $q\leq a<b$.
\end{itemize}
So from these three cases we see that $R(\tfrac{a}{b})$ is a non-negative rational number $\tfrac{p}{q}\in [0,1]$, with $q<b$. Hence in finitely many steps, say $n$, we must have that $R^n(\tfrac{a}{b}) \in\{ 0,1\}$.
\end{proof}

\begin{remark}\label{rem:Romikexpansionrationalnumbers}{\rm 
As each of the tree branches of $R$ is \emph{full}, and $0$ and $1$ are fixed points of $R$, we see that rational numbers $p/q$ whose orbit eventually ``end'' in $0$ must be a pre-image under $R$ of $\tfrac{1}{2}$, while those ``ending'' in $1$ are the pre-images of $\tfrac{1}{3}$.
}\hfill $\triangle$
\end{remark}

\begin{proposition}\label{prop:relationRomikGauss}
If $x\in(0,1)$ has RCF-expansion $x=[0;a_1,a_2,a_3,a_4,\dots]$, then
\begin{equation}\label{RelationFG}
R(x) = \begin{cases}
[0;a_1-2,a_2,a_3,a_4,\dots], & \text{if $a_1>2$}; \\
[0;a_2,a_3,a_4\dots], & \text{if $a_1=2$};\\
[0;a_3+1,a_4,\dots], & \text{if $a_1=1$ and $a_2=1$}; \\
[0;1,a_2-1,a_3,a_4,\dots], & \text{if $a_1=1$ and $a_2\geq 2$}; 
\end{cases}
\end{equation}
\end{proposition}

\begin{proof} 
It is shown in~\cite{[DKSS]} that if $x\in(0,1)$ has RCF-expansion $x=[0;a_1,a_2,a_3\dots]$, then
\begin{equation*}
F(x) = \begin{cases}
[0;a_1-1,a_2,a_3,\dots], & \text{if $a_1>1$}; \\
[0;a_2,a_3,\dots], & \text{if $a_1=1$},
\end{cases}
\end{equation*}
and from this one see that $(G^n(x))_{n\geq 0}$ forms a subsequence of $(F^n(x))_{n\geq 0}$, which also immediately follows from~\eqref{jumpFtoR}. 

From ~\eqref{relationRomikFarey}, it immediately follows  that
$$
R(x) = F^2(x) = [0;a_1-2,a_2,a_3,\dots],\quad  \text{if $a_1>2$, so if $x\in (0,\tfrac{1}{3}]$},
$$
and from definition~\eqref{GaussMap} of the Gauss map $G$ we immediately find that
$$
R(x) = G(x) = [0;a_2,a_3,\dots], \quad \text{if $a_1=2$, so if $x\in (\tfrac{1}{3},\tfrac{1}{2}]$}.
$$

For $x\in (\tfrac{1}{2},1)$ we have two cases, depending on whether $x\geq \tfrac{2}{3}$ (i.e., $a_2\geq 2$), or whether $x\in (\tfrac{1}{2},\tfrac{2}{3})$ (i.e., $a_2=1$). In the first case, we need an \emph{insertion}, while in the second case we need an \emph{singularization}; see~\cite{[DHKM]} where both insertions and singularations are used to get \emph{flipped} continued fraction expansions. In~\cite{[KL]} both singularizations and insertions were used to give on ``on-line'' algorithm to convert the RCF-expansion of any $x\in [0,1)$ into the ECF-expansion of $x$; see also Remark 2 in~\cite{[KL]} for full details.\smallskip\

\emph{Insertions} are based on the following observation (here $a,b\in\N$, $b\geq 2$, $\xi\in [0,1)$),
\begin{equation}\label{insertion}
a + \frac{1}{b + \xi} = a + 1 + \frac{{\color{red} -1}}{{\color{red} 1} + \displaystyle \frac{1}{b - 1 + \xi}}.
\end{equation}
(we have inserted ``{\color{red} $-1/1$}'' before $b$), while \emph{singularizations} are based on the following (here $a,b\in\N$, $b\geq 2$, $\xi\in [0,1)$),
\begin{equation}\label{singularization}
a + \frac{{\color{red} 1}}{{\color{red} 1}+ \displaystyle \frac{1}{b + \xi}} = a + 1 + \frac{-1}{b + 1 + \xi}
\end{equation}
(we have singularized {\color{red} ``$1/1$''}).\smallskip\

Now let $x\in (\tfrac{1}{2},1)$, then $x=[0;1,a_2,a_3,a_4,\dots]$, and therefore $G(x)=\tfrac{1}{x}-1 = [0;a_2,a_3,a_4,\dots]$. But then we have,
\begin{equation}\label{FromRomikToRCF}
R(x) = 1 - \left( \frac{1}{x}-1\right) = 1 - [0;a_2,a_3,a_4,\dots] = 1 - \frac{1}{a_2 + \displaystyle \frac{1}{a_3+\ddots}}.
\end{equation}
If $a_2\geq 2$, we insert in equation~\eqref{FromRomikToRCF} ``$-1/1$'' before $a_2$ in $[0;a_2,a_3,\dots]$, to find
$$
R(x) = 1 - \left( 1 + \frac{-1}{1+ \displaystyle\frac{1}{a_2-1+\displaystyle \frac{1}{a_3+\ddots}}}\right) = [0; 1, a_2-1,a_3,\dots].
$$
Next, if $a_2=1$, we singularize in~\eqref{FromRomikToRCF} the partial quotient $a_2=1$ to arrive at
$$
R(x) = 1 - \left( 1 + \frac{-1}{a_3+1+\displaystyle \frac{1}{a_4+\ddots}} \right) = [0; a_3+1,a_4,\dots].
$$
\end{proof}

\subsection{Quadratic irrational numbers and periodic expansions}\label{sec:quadraticirrational}
It is a classic result that if $x$ is a real \emph{quadratic irrational} number, i.e., a number of the form
$$
x = \frac{a+\sqrt{b}}{c},
$$
where $a,b,c\in\Z$, $b\geq 1$ and $b\neq \Box$ (i.e., $b$ is not a square), and $c\neq 0$, then the RCF-expansion of $x$ (and therefore also the orbit $(G^n(x-a_0))_{n\geq 0}$ of $x-a_0$) is \emph{eventually periodic}
\begin{equation}\label{periodicRCFexpansion}
x = [a_0;a_1,\dots,a_{\ell}, \overline{a_{\ell+1},\dots,a_{\ell+p}}],
\end{equation}
where $a_0\in\Z$ is such that $a_0=\lfloor x\rfloor$, where $a_1,\dots,a_{\ell}$ is the so-called \emph{pre-period}, and where the bar indicates the \emph{period} $a_{\ell+1},\dots,a_{\ell+p}$. Here $p\in\N$ is the \emph{period length}. Conversely, if the RCF-expansion of a real number $x$ is given by~\eqref{periodicRCFexpansion} (i.e., if the RCF-expansion of $x$ is eventually periodic), then $x$ is a quadratic irrational. For a proof of these results, see e.g.~\cite{[HW]}.

\begin{proposition}\label{prop:eventuallyperiodicorbitofR}
Let $x\in\R$ be a quadratic irrational number with RCF-expansion given by~\eqref{periodicRCFexpansion}. Then the orbit $(R^n(x-a_0))_{n\geq 0}$ of $x-a_0$ under the Romik map $R$ is eventually periodic. 
\end{proposition}

\begin{remark}\label{periodicity}{\rm Without loss of generality, we may assume in Proposition~\ref{prop:eventuallyperiodicorbitofR} that $a_0=0$; i.e., that $x\in [0,1)$. It follows from Proposition~\ref{prop:relationRomikGauss} that $R(x)$ (and therefore also $R^n(x)$ for each $n\in\N$) has an RCF-expansion that is eventually periodic, and therefore $R^n(x)$ is a quadratic irrational number for all $n\in\N$ (this also follows immediately from the definition of the Romik map $R$). It also follows from repeated application of Proposition~\ref{prop:relationRomikGauss} that there exist an $N\in\N$, $N\geq \ell$, and a $k\geq 1$, such that the partial quotients in the pre-period have ``disappeared.'' However, this does \textbf{not} imply that $R^N(x)$ is of the form
$$
R^N(x) = [0; \overline{a_{\ell+k}, a_{\ell+k+1}a_{\ell+k+2},\dots, a_{\ell+k+p-1}}],
$$
as the following example shows. Let $x=[0;1,\bar{2}]=\tfrac{1}{2}\sqrt{2}=0.707106\cdots$. Then $R(x)=[0;a_1=1, a_2-1=1, a_3=2, a_4=2,\dots]=[0;1,1,\bar{2}]=2-\sqrt{2}=0.585786\cdots$, $R^2(x)=[0;a_3+1=3,a_4=2,\dots]=[0;3,\bar{2}]=\frac{2-\sqrt{2}}{2}=0.29289\cdots$, and $R^3(x)=[0;a_3+1-2=1,a_4=2,a_5=2,\dots]=[0;1,\bar{2}]=x$, and everything repeats itself periodically.}\hfill $\triangle$


\end{remark}

\begin{proof}
In view of Remark~\ref{periodicity} we may assume that $x\in [0,1)$, and that $x$ is either purely periodic, i.e., 
$$
x = x_0 = [0;\overline{a_1,\dots,a_p}],
$$
or that $x$ is of the form
$$
x = \xi_1 = [0;1,1,\overline{a_1,\dots,a_p}],
$$
where in both cases we assume that the \emph{period length} $p\geq 1$ is minimal; cf.~\eqref{periodicRCFexpansion}. Now set for $i=1,2,\dots,p$: $x_i = G^i(x_0)$, so 
$$
x_1=G(x_0)=[0;\overline{a_2,\dots,a_p,a_1}],\,\, \dots ,\,\, x_{p-1}=G(x_{p-2})=[0;\overline{a_p,a_1,\dots,a_{p-1}}], 
$$
and $x_p=x_0$, and
$$
\xi_i=[0;1,1,\overline{a_i,\dots,a_p,a_1,\dots,a_{i-1}}],\,\,\text{for $i=2,\dots,p$}.
$$
Set $\mathcal{P}=\{ x_0,\dots,x_{p-1},\xi_1,\dots,\xi_p\}$, and for $0\leq  i\leq p-1$, let  $M_i\in\N$ be the minimal positive integer for which $R^{M_i}(x_i)\in\mathcal{P}$. Also, for $1\leq j\leq p$, let $N_j\geq 1$ be the minimal positive integer for which $R^{N_j}(\xi_j)\in\mathcal{P}$. Now define the map $P:\mathcal{P}\to\mathcal{P}$ by
$$
P(x_i)=R^{M_i}(x_i),\quad \text{for $i=0,1,\dots,p-1$},
$$
and
$$
P(\xi_j)=R^{N_j}(\xi_j),\quad \text{for $j=1,2,\dots,p$}.
$$
Then the sequence $(P^n(x_0))_{n\geq 0}$ is a sequence in $\mathcal{P}$ and a subsequence of the orbit $(R^n(x))_{n\geq 0}$. Since the set $\mathcal{P}$ is finite, we must have that there exist integers $0\leq r<s$, such that $P^r(x)=P^s(x)$, which immediately implies that the orbit of $x$ under $R$ is eventually periodic.
\end{proof}

\begin{example}\label{ex:periodicexpansion}{\rm 
As an example, let $x=g/3.938=0.156941084\cdots = [0;\overline{6,2,1,2,4,1,1}]$, where $g=\frac{\sqrt{5}-1}{2}=0.618033988\cdots$ is\footnote{There are two \emph{golden means}: $g$ and $G=g+1=1/g$.} the \emph{golden mean}. Using the well-know recursion relations for the numerator $P_n$ and denominator $Q_n$ of the RCF-convergents of $x$ we find the first RCF-convergents of $x$ in Table~\ref{table1} below.
\begin{table}
$$
\begin{array}{|c|cc|c|c|c|c|c|c|c|}
\hline
n= & -1 & 0 & 1 & 2 & 3 & 4 & 5 & 6 & 7\\
\hline
a_n=  &  & 0 & a_1=6 & 2 & 1 & 2 & 4 & 1 & 1 \\
\hline
P_n= & 1 & 0 & 1 & 2 & 3 & 8 & 35 & 43 & 78\\
Q_n= & 0 & 1 & 6 & 13 & 19 & 51 & 223 & 274& 497\\
\hline
\tfrac{P_n}{Q_n}= & \infty & 0 & \tfrac{1}{6} & \tfrac{2}{13} & \tfrac{3}{19} & \tfrac{8}{51} & \tfrac{35}{223} & {\phantom{X}}\tfrac{43}{274}^{\phantom{X}}_{\phantom{Y}} & \tfrac{78}{497}\\
\hline
\end{array}
$$
\caption{The first RCF-convergents of $x=g/3.938$}\label{table1}
\end{table}
To make an add for the quality of approximation of the RCF-convergents to $x$, note that
$$
Q_7^2\left| x-\tfrac{P_7}{Q_7}\right| = 497^2\left| x - \tfrac{78}{497}\right| = 0.139784885\cdots .
$$

Now using~\eqref{RelationFG} from Proposition~\ref{prop:relationRomikGauss}, we easily find the $R$-orbit of $x$, where the points in the orbit are expressed by their RCF-expansion; see Table~\ref{table2}. 
\begin{table}[h]
$$
\begin{array}{llllll}
R(x) &=& [0;4,\overline{2,1,2,4,1,1,6}], & R^2(x) &=&[0;2,\overline{2,1,2,4,1,1,6}],\\
R^3(x) &=& [0;\overline{2,1,2,4,1,1,6}], & R^4(x) &=& [0;\overline{1,2,4,1,1,6,2}], \\
R^5(x) &=& [0;1,1,\overline{4,1,1,6,2,1,2}],& R^6(x) &=& [0;5,\overline{1,1,6,2,1,2,4}],\\
R^7(x) &=& [0;3,\overline{1,1,6,2,1,2,4}], & R^8(x) &=& [0;1,\overline{1,1,6,2,1,2,4}], \\
R^9(x) &=& [0;2,\overline{6,2,1,2,4,1,1}], & R^{10}(x) &=& [0;\overline{6,2,1,2,4,1,1}]\,\, =\,\, x.\\
\end{array}
$$
\caption{The beginning of the $R$-orbit of $x=g/3.938$}\label{table2}
\end{table}

We see that the orbit $(R^n(x))_{n\geq 0}$ of $x$ under $R$ is purely periodic, with period length $10$. For our example $x=g/3.938$ one finds that
$$
R(x)=F^2(x),\, R^2(x)=F^4(x),\, R^3(x)=F^6(x) (=G(x)),\, R^4(x)=F^8(x) (=G^2(x))
$$
which follow from~\eqref{relationRomikFarey}, and
$$
R^8(x)=F^{14}(x),\,\, R^{10}(x) = F^{17}(x)(=G^7(x)=x).
$$
So in the sequence $(F^n(x))_{n\geq 0}$ the terms $R^5(x)$, $R^6(x)$, $R^7(x)$, and $R^9(x)$ are missing; cf.~\eqref{relationRomikFarey}. Clearly, the sequence $(R^n(x))_{n\geq 0}$ is in general \textbf{not} a subsequence of $(F^n(x))_{n\geq 0}$. Similarly, the terms $G^i(x)$,  for $i=3,4,5,6$, are \textbf{not} in the sequence $(R^n(x))_{n\geq 0}$ (recall that for all $x\in[0,1]\setminus\Q$, the sequence $(G^n(x))_{n\geq 0}$ is a subsequence of the sequence $(F^n(x))_{n\geq 0}$).}\hfill $\triangle$
\end{example}

\begin{remark}\label{periodicityECFandOOCF}{\rm
As Kim \emph{et al.}\ showed in~\cite{[KLL]} that both the ECF-expansion and the OOCF-expansion can be obtained from the Romik map as suitable \emph{jump transformations}, an immediate consequence of Proposition~\ref{prop:eventuallyperiodicorbitofR}, is that if $x$ is a quadratic irrational, the ECF-expansion and the OOCF-expansion of $x$ are eventually periodic. The converse, which holds for the RCF (``if the RCF-expansion of $x$ is eventually periodic, then $x$ is quadratic irrational''), does not hold here, as we saw that there are rationals whose OOCF-expansion is infinite. See Proposition~\ref{prop:RationalsHaveFiniteOrbit} and Remark~\ref{rem:Romikexpansionrationalnumbers}, and in particular, see Theorem~1.3 in~\cite{[KLL]}.}\hfill $\triangle$
\end{remark}

\subsection{The natural extension for the Romik map}\label{subsec:naturalextension}
In fact, Romik already gave in~(11) (\cite{[R]}, p.~6056) the inverse branches of $R$. For $t\in [0,1]$, the inverse  $S_{\ell}$ of the left branch of $R$,  the inverse $S_m$ of the middle branch of $R$, and the inverse $S_r$ of the right branch of $R$ are, respectively, given for $y\in [0,1]$ by
\begin{equation}\label{naturalextensionmapR}
\begin{array}{ccc}
S_{\ell}(y) & = & \frac{y}{1+2y} \in [0,\tfrac{1}{3}];\medskip\ \\
S_m(y) & = & \frac{1}{2+y} \in [\tfrac{1}{3}, \tfrac{1}{2}];\medskip\ \\
S_r(y) & = &\frac{1}{2-y}\in [\tfrac{1}{2}, 1].
\end{array}
\end{equation}
Therefore, from~\eqref{RomikMapR} and~\eqref{naturalextensionmapR} a ``natural candidate'' for the natural extension map $\mathcal{R}:[0,1]\times [0,1]\to [0,1]\times [0,1]$ is given by:
\begin{equation}\label{naturalextensionmapR2}
\mathcal{R}(x,y) = \begin{cases}
\left( \frac{x}{1-2x},\frac{y}{1+2y}\right), & \text{if $x\in [0,\tfrac{1}{3}]$}; \medskip\ \\
\left( \frac{1}{x}-2, \frac{1}{y+2}\right), & \text{if $x\in [\tfrac{1}{3},\tfrac{1}{2}]$}; \medskip\ \\
\left( 2-\frac{1}{x}, \frac{1}{2-y}\right), & \text{if $x\in [\tfrac{1}{2},1]$}.
\end{cases}
\end{equation}
See also Figure~\ref{figurenatextandimageunderR}, which gives a course ``picture'' of the dynamics of the map $\mathcal{R}$. In Remark~\ref{rem:OOCF} we give a more detailed image of the dynamics of $\mathcal{R}$.\medskip\

\begin{figure}[h]
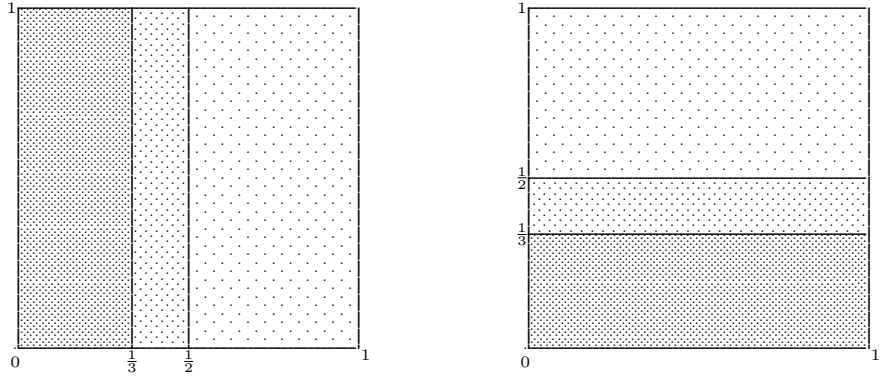

$$
\beginpicture
    \setcoordinatesystem units <0.45cm,0.45cm>
    \setplotarea x from 0 to 26, y from 0 to 8
  \put {\tiny{$0$}} at -0.1 -0.4
  \put {\tiny{$0$}} at 14.9 -0.4
  \put {\tiny{$1$}} at -0.2 10
  \put {\tiny{$1$}} at 10.2 -0.2
  \put {\tiny{$1$}} at 25.2 -0.2
  \put {\tiny{$1$}} at 14.8 10
  \put {\tiny{$\tfrac{1}{3}$}} at 3.333333 -0.4
  \put {\tiny{$\tfrac{1}{3}$}} at 14.8 3.333333
  \put {\tiny{$\tfrac{1}{2}$}} at 5 -0.4
  \put {\tiny{$\tfrac{1}{2}$}} at 14.8 5
\setlinear \plot
0 0 10 0 
/
\setlinear \plot
0 0 0 10
/
\setlinear \plot
0 10 10 10
/
\setlinear \plot
10 0 10 10
/
\setlinear \plot
15 0 25 0
/
\setlinear \plot
15 0 15 10
/
\setlinear \plot
15 10 25 10
/
\setlinear \plot
25 0 25 10
/
\setlinear \plot
5 0 5 10
/
\setlinear \plot
3.333333 0 3.333333 10
/
\setlinear \plot
15 5 25 5
/
\setlinear \plot
15 3.333333 25 3.333333
/
\setshadegrid span <1.15pt>
\vshade 0 0 10
3.333333 0 10
/
\setshadegrid span <1.15pt>
\vshade 15 0 3.333333
25 0 3.333333
/
\setshadegrid span <2pt>
\vshade 3.33333 0 10
5 0 10
/
\setshadegrid span <2pt>
\vshade 15 3.333333 5
25 3.333333 5
/
\setshadegrid span <3.2pt>
\vshade 5 0 10
10 0 10
/
\setshadegrid span <3.2pt>
\vshade 15 5 10
25 5 10
/


\endpicture
$$ \caption[naturalextension]{Left: the space of the natural extension $[0,1]\times [0,1]$. Right: the image of $[0,1]\times [0,1]$ under $\mathcal{R}$.} \label{figurenatextandimageunderR}
\end{figure}

Let $x\in(0,1)$ has RCF-expansion $x=[0;a_1,a_2,\dots]$, then in Proposition~\ref{prop:relationRomikGauss} we expressed the RCF-expansion of $R(x)$ in terms of the RCF-expansion of $x$. Let $\pi_2$ be the projection of a 2-dimensional vector $(x,y)$ on its second coordinate (i.e., $\pi_2(x,y)=y$, for all $x,y\in\R$), then we define the second coordinate map $S$ of $\mathcal{R}$ for $y\in [0,1]$ as
$$
S(y) = \pi_2(\mathcal{R}(x,y)),\quad \text{for $(x,y)\in [0,1]^2$}.
$$
We have the following result, which shows that the second coordinate map $S$ of $\mathcal{R}$ truly behaves as the inverse of $R$.
\begin{proposition}\label{prop:seconcoordinaterelationRomikGauss}
If $x\in(0,1)$ has RCF-expansion $x=[0;a_1,a_2,\dots]$ and $y=[0;b_1,b_2,\dots]$, then
\begin{equation}\label{RelationFG2}
S(y) = \begin{cases}
[0;b_1+2,b_2,\dots], & \text{if $a_1>2$}; \\
[0;2,b_1,b_2,\dots], & \text{if $a_1=2$};\\
[0;1,b_2+1,b_3,\dots], & \text{if $a_1=1$, and $b_1=1$}; \\
[0;1,1,b_1-1,b_2,\dots], & \text{if $a_1=1$, and $b_1\geq 2$}.
\end{cases}
\end{equation}
\end{proposition}

\begin{proof}
If $y=[0;b_1,b_2,\dots]$, then the first two cases are straightforward. E.g., if $a_1>2$,
$$
\frac{y}{1+2y} = \frac{1}{2+\frac{1}{y}} = \frac{1}{2+b_1+\displaystyle \frac{1}{b_2+\displaystyle \frac{1}{b_3+\ddots}}} = [0;2+b_1,b_2,b_3,\dots].
$$
The last two cases again involve (the inverse of) an insertion and (the inverse of) a singularization. If $a_1=1$ and $b_1=1$, then
$$
\frac{1}{2-y} = \frac{1}{2-\displaystyle\frac{1}{1+\displaystyle\frac{1}{b_2+\ddots}}}.
$$
We can recognize an insertion here, and we find (``undoing'' the insertion):
$$
\frac{1}{2-y} = \frac{1}{1+\displaystyle\frac{1}{b_2+1+\displaystyle\frac{1}{b_3+\ddots}}} = [0;1,b_2+1,b_3,\dots].
$$
If $a_1=1$ and $b_1>1$, then (``undoing'' a singularization):
\begin{eqnarray*}
\frac{1}{2-y} &=& \frac{1}{2-\displaystyle\frac{1}{b_1+\ddots}} = \frac{1}{1+\displaystyle\frac{1}{1+\displaystyle\frac{1}{b_1-1+\displaystyle\frac{1}{b_2+\ddots}}}}\\
&=& [0;1,1,b_1-1,b_2,\dots].
\end{eqnarray*}
\end{proof}

\begin{proposition}\label{propertiesnaturalextensionRomikmap}
Apart from a Lebesgue set of measure 0, the Romik map $\mathcal{R}: [0,1]^2\to [0,1]^2$ defined in~\eqref{naturalextensionmapR2}, is bijective. The absolutely continuous invariant measure $\bar{\mu}$ of $\mathcal{R}$ in $[0,1]^2$ 
is given by
\begin{equation}\label{densityf}
\bar{\mu}(D)=\int_D \frac{1}{(x+y-2xy)^2}\, {\rm d}(\lambda\times\lambda)
\end{equation}
where $D$ is a Borel set in $[0,1]^2$, and $\mathcal{R}$ is ergodic. The map $\mathcal{R}$ is the natural extension of $R$.
\end{proposition}

\begin{proof}
The fact that the Romik map $\mathcal{R}: [0,1]^2\to [0,1]^2$ is bijective (apart from a Lebesgue set of measure 0) follows directly from~\eqref{stepsRomikmap1}, \eqref{stepsRomikmap2}, and~\eqref{naturalextensionmapR}. To see this, note that if $x\in [\tfrac{1}{n+1},\tfrac{1}{n}]$, for $n\geq 3$, we have~\eqref{stepsRomikmap1} and~\eqref{naturalextensionmapR}
$$
\begin{array}{ccc}
\mathcal{R}\left( [\tfrac{1}{n+1},\tfrac{1}{n}]\times [0,1]\right) &=& [\tfrac{1}{n-1},\tfrac{1}{n-2}]\times [0,\tfrac{1}{3}], \\
\mathcal{R}\left( [\tfrac{1}{3},\tfrac{1}{2}]\times [0,1]\right) &=& [0,1]\times [\tfrac{1}{3},\tfrac{1}{2}],
\end{array}
$$
and $\mathcal{R}\left( [\tfrac{1}{2},1]\times [0,1]\right) = [0,1]\times [\tfrac{1}{2},1]$.\medskip\

We use a standard argument to show that the invariant measure $\bar{\mu}$ given by~\eqref{densityf}. Since $\mathcal{R}$ is bijective a.s., it is enough to show that $\bar{\mu} (D) = \bar{\mu} (\mathcal{R}(D))$ for any $D\in\mathcal{B}([0,1]^2)$. We do this using the \emph{change of variable formula}, and show
$$
\iint_D \frac{1}{(x+y-2xy)^2}\, {\rm d}x\,{\rm d }y = \iint_{{\mathcal R}(D) }\frac{1}{(\xi+\eta-2\xi\eta)^2}\, {\rm d}\xi\,{\rm d }\eta ,
$$
where $(\xi ,\eta ) = \mathcal{R} (x,y)$. Due to definition~\eqref{naturalextensionmapR2} of $\mathcal{R}$ we need to calculate this in three cases.\medskip\

\textbf{Case 1}: let $x\in [0,\tfrac{1}{3})$  and $(\xi ,\eta ) = \mathcal{R} (x,y) = \left( \tfrac{x}{1-2x}, \tfrac{y}{1+2y}\right)$. So $\xi = \tfrac{x}{1-2x}$ and $\eta = \tfrac{y}{1+2y}$. Solving for $x$ and $y$ we get $x=\tfrac{\xi}{1+2\xi}$ and $y=\tfrac{\eta}{1-2\eta}$.\smallskip\

Now we calculate the Jacobian:
$$
J = \left|
\begin{array}{cc}
\frac{1}{(1+2\xi )^2} & 0 \\
0 & \frac{1}{(1-2\eta )^2}
\end{array}
\right| = \frac{1}{(1+2\xi )^2}\frac{1}{(1-2\eta )^2}.
$$
Thus,
\begin{eqnarray*}
{\bar{\mu}} (D) &=& \iint_D \frac{1}{(x+y-2xy)^2}\, {\rm d}x\,{\rm d }y\\
&=&  \iint_{\mathcal{R}(D)} \frac{1}{\left(  \tfrac{\xi}{1+2\xi}+\tfrac{\eta}{1-2\eta}-2\tfrac{\xi}{1+2\xi}\tfrac{\eta}{1-2\eta}\right)^2}\, \frac{1}{(1+2\xi )^2}\frac{1}{(1-2\eta )^2}\, {\rm d}\xi\,{\rm d }\eta \\
&=& \iint_{\mathcal{R}(D)} \frac{1}{(\xi + \eta -2\xi\eta)^2}\, {\rm d}\xi\,{\rm d }\eta \\
&=& {\bar{\mu}} (\mathcal{R}(D)).
\end{eqnarray*}

\textbf{Case 2}: let $x\in [\tfrac{1}{3},\tfrac{1}{2})$  and $(\xi ,\eta ) = \mathcal{R} (x,y) = \left(\tfrac{1}{x}-2, \tfrac{1}{y+2}\right)$, so $\xi = \tfrac{1}{x}-2$ and $\eta = \tfrac{1}{y+2}$, implying $x = \tfrac{1}{\xi +2}$ and $y=\tfrac{1-2\eta}{\eta}$. Now, the Jacobian $J$ satisfies:
$$
J = \left|
\begin{array}{cc}
- \frac{1}{(2+\xi )^2} & 0\\
0 & -\frac{1}{\eta^2}
\end{array}
\right| = \frac{1}{(2+\xi )^2\, \eta^2}.
$$
But then we have
\begin{eqnarray*}
{\bar{\mu}} (D) &=& \iint_D \frac{1}{(x+y-2xy)^2}\, {\rm d}x\,{\rm d }y\\
&=&  \iint_{\mathcal{R}(D)} \frac{1}{\left( \tfrac{1}{\xi +2}+\tfrac{1-2\eta}{\eta} - 2\tfrac{1}{\xi +2}\tfrac{1-2\eta}{\eta}\right)^2}\, \tfrac{1}{(\xi +2)^2\, \eta^2} {\rm d}\xi\,{\rm d }\eta\\
&=& \iint_{\mathcal{R}(D)} \frac{1}{(\xi + \eta -2\xi\eta)^2}\, {\rm d}\xi\,{\rm d }\eta \\
&=& {\bar{\mu}} (\mathcal{R}(D)).
\end{eqnarray*}

\textbf{Case 3}: let $x\in [\tfrac{1}{2},1]$  and $(\xi ,\eta ) = \mathcal{R} (x,y) = \left( 2-\tfrac{1}{x}, \tfrac{1}{2-y}\right)$. Now $x=\tfrac{1}{2-\xi}$ and $y=\tfrac{2\eta-1}{\eta}$, and we find Jacobian $J$ in this case:
$$
J = \left|
\begin{array}{cc}
\frac{1}{(2-\xi )^2} & 0\\
0 & \frac{1}{\eta^2}
\end{array}
\right| = \frac{1}{(2-\xi )^2\, \eta^2}.
$$
So,
\begin{eqnarray*}
\bar{\mu} (D) &=& \iint_D \frac{1}{(x+y-2xy)^2}\, {\rm d}x\,{\rm d }y\\
&=&  \iint_{\mathcal{R}(D)} \frac{1}{\left( \tfrac{1}{2-\xi} + \tfrac{2\eta-1}{\eta} - \tfrac{2(2\eta-1)}{(2-\xi)\eta}\right)^2}\, \tfrac{1}{(2-\xi)^2\,\eta^2}{\rm d}\xi\,{\rm d }\eta\\
&=& \iint_{\mathcal{R}(D)} \frac{1}{(\xi + \eta -2\xi\eta)^2}\, {\rm d}\xi\,{\rm d }\eta \\
&=& \bar{\mu} (\mathcal{R}(D)).
\end{eqnarray*}
In all three possible cases, we find that $\bar{\mu }(D) =\bar{ \mu }(\mathcal{R}(D))$. I.e., ${\bar{\mu}}$ is $\mathcal{R}$-invariant.\smallskip\

We show that the projection of ${\bar{\mu}}$ on the first coordinate gives the invariant measure $\mu$ of the Romik map. To see this, let $A$ be any Borel set in $[0,1]$, we have
\begin{eqnarray*}
\bar{\mu}(A\times [0,1]
)&=&\int_A\int_0^1 \frac{1}{(x+y-2xy)^2}\, {\rm d }y\, {\rm d }x\,\,=\,\, \int_A\frac{1}{1-2x} \int_x^{1-x} \frac{1}{u^2}\, {\rm d}u {\rm d}x\\
&=& \int_A\left[\frac{1}{1-2x} \frac{-1}{u} \right]_{u=x}^{u=1-x}{\rm d}x\,\, =\,\, \int_A \frac{1}{1-2x} \left( \frac{-1}{1-x} + \frac{1}{x}\right)\, {\rm d}x\\
&=& \int_A \frac{1}{x(1-x)}\, {\rm d}x.
\end{eqnarray*}

Ergodicity follows from the fact that the one-dimensional system is ergodic (cf.\ Theorem~3 in Romik's paper~\cite{[R]}), and that $([0,1]^2,\mathcal{B},\mu,\mathcal{R})$ is the (rather ``a'') natural extension of the one-dimensional system (i.e., an almost surely bijective system which has the original system as a factor) is shown in a standard way. For more information on natural extensions, see Rokhlin's celebrated paper~\cite{[Rok]}.
\end{proof}

\begin{remark}\label{rem:OOCF}\rm{
As mentioned in Section~\ref{sec:RomikMap}, in~\cite{[KLL]} the continued fraction map $T_{\rm OOCF}$ of the OOCF was obtained from the Romik map as a \emph{jump transformation}. In~Proposition~2.2 of~\cite{[KLL]} it was also established that $T_{\rm OOCF}$ is ergodic, and that the $T_{\rm OOCF}$-invariant measure has density $\tfrac{1}{x}$ on $[0,1]$. These results also follow from a suitable \emph{induced transformation} of the Romik map $\mathcal{R}$, as we now will outline briefly. The advantage of doing so, is not only that we re-prove the result by Kim \emph{et al.} from~\cite{[KLL]}, we also find the (in fact \emph{a} version of the) natural extension of the OOCF; see also Section~3 in~\cite{[BY]}, where a version of the natural extension of the RCF was obtained from Ito's  natural extension of the Farey-map.\bigskip\

Let $\mathcal{O}=[0,1]\times [0,\tfrac{1}{2}]$, then clearly $\mu (\mathcal{O})>0$ (in fact, $\mu (\mathcal{O})=\infty$), and define for $(x,y)\in\mathcal{O}$
the \emph{return time} $n(x,y)$ \emph{of} $(x,y)$ \emph{to} $\mathcal{O}$ as
$$
n(x,y) = \inf \{ n\in\N\, |\, \mathcal{R}^n(x,y)\in\mathcal{O}\} .
$$
Note that $n(x,y)$ is finite for Lebesgue a.e.\ $(x,y)\in\mathcal{O}$. Next, define the \emph{induced map} $\mathcal{I}_{\mathcal{O}}:\mathcal{O}\to \mathcal{O}$ by
$$
\mathcal{I}_{\mathcal{O}}(x,y) = \mathcal{R}^{n(x,y)}(x,y),\quad \text{for $(x,y)\in \mathcal{O}$}.
$$
Although $\mathcal{O}$ has infinite $\mu$ measure, since $\mathcal{R}$ is invertible, it is easily seen that if $J$ is any rectangle in $\mathcal{O}$ then $\mu(J)=\mu(\mathcal{I}_{\mathcal{O}}(J))$. Thus, on $\mathcal{O}$ the density of the induced map $\mathcal{I}_{\mathcal{O}}$ is the density of the $\mathcal{R}$-invariant $\sigma$-finite, infinite measure $\mu$, given by
$$
\frac{1}{(x+y-2xy)^2},\quad (x,y)\in \mathcal{O}.
$$
To see that $\mathcal{I}_{\mathcal{O}}$ is also ergodic we can adapt the proof for the finite (i.e., probability measure) case (which is an exercise in~\cite{[DK]}). Let $B\subset \mathcal{O}$ be a measurable set which is $\mathcal{I}_{\mathcal{O}}$-invariant. In order to prove ergodicity we now must show that either $\mu (B)=0$ or $\mu (\mathcal{O}\setminus B) = 0$. Define $\mathcal{O}_n=\{ (x,y)\in \mathcal{O}\, |\, n(x,y)=n\}$, and
$$
C_n=\{ (x,y)\in [0,1]\times [0,1]\setminus\mathcal{O}\,\, |\,\, \mathcal{R}(x,y)\not\in \mathcal{O},\dots , \mathcal{R}^{n-1}(x,y)\not\in \mathcal{O}, \mathcal{R}^n(x,y)\in \mathcal{O}\}.
$$
Furthermore, let 
$$
F = \bigcup_{n=1}^{\infty} (C_n\cap \mathcal{R}^{-n}(B)),\quad \text{and $D=F\cup B$}. 
$$
Using the fact that $B=\mathcal{I}^{-1}_{\mathcal{O}}(B)=\bigcup_{n=1}^{\infty} (\mathcal{O}_n\cap \mathcal{R}^{-n}(B))$,
$\mathcal{R}^{-1}(\mathcal{O})=\mathcal{O}_1\cup C_1$, and $\mathcal{R}^{-1}(C_n)=\mathcal{O}_{n+1}\cup C_{n+1}$, it is easy to see $\mathcal{R}^{-1}(D)=D$.

By ergodicity of $\mathcal{R}$, we have $\mu (D)=0$ or $\mu ([0,1]\times [0,1]\setminus D)=0$. If $\mu (D)=0$, we obviously have $\mu (B)=0$. Let us therefore assume that $\mu ([0,1]\times [0,1]\setminus D)=0$. Since $F\subset \mathcal{O}^c$ and $B\subset \mathcal{O}$ are disjoint, we have
$$
\mathcal{O}\setminus B \subset (\mathcal{O}\setminus B)\cup (\mathcal{O}^c\setminus F) = [0,1]\times [0,1]\setminus (B\cup F) = [0,1]\times [0,1]\setminus D,
$$
and it follows that $\mu (\mathcal{O}\setminus B) = 0$.\bigskip\

Now let $E$ be the first coordinate-map of $\mathcal{I}_{\mathcal{O}}$. I.e., for $(x,y)\in \mathcal{O}$, we have
$$
E(x) = \pi_1(\mathcal{I}_{\mathcal{O}}(x,y)),
$$
where $\pi_1(x,y)=x$ is the projection on the first coordinate of the vector $(x,y)$. We \textbf{claim} that
\begin{equation}\label{claim}
T_{\rm OOCF}(x) = E(x),\quad \text{for $x\in [0,1]$}.
\end{equation}
To see this claim holds we consider three cases, where we assume that $(x,y)\in \mathcal{O}$, $k=\inf\{ n\in\N\cup \{ 0\}\, |\, R^n(x)\in [0,\tfrac{1}{2}]\}$, and $\ell = n(x,y) =\inf\{ n\in\N\, |\, \mathcal{R}^n(x,y)\in \mathcal{O}\}$. As $0\leq y\leq \tfrac{1}{2}$, from~\eqref{naturalextensionmapR2} it immediately follows that the value of $\ell$ \emph{only} depends on $x$.
\begin{itemize}
\item[($i$)] Let $0\leq x\leq \tfrac{1}{3}$. Then there exist $n,m\in\N$, $n\geq 3$, such that 
$$
\tfrac{1}{n+1}\leq x < \tfrac{1}{n},\quad \text{and}\,\, \tfrac{1}{m+1}\leq y\leq \tfrac{1}{m}.
$$
From~\eqref{Rimageinterval} and~\eqref{naturalextensionmapR2} it follows that 
$$
\mathcal{R}\left( [\tfrac{1}{n+1},\tfrac{1}{n}]\times [\tfrac{1}{m+1},\tfrac{1}{m}]\right) = [\tfrac{1}{n-1},\tfrac{1}{n-2}]\times [\tfrac{1}{m-1},\tfrac{1}{m-2}],
$$
so the point $(x,y)$ ``moves two intervals to the right and two down''; see Figure~\ref{case(i)}. Clearly $\mathcal{R}(x,y)\in \mathcal{O}$, and we immediately have in this case that $k=0$, so $k+1=1=\ell$, and therefore~\eqref{claim} holds,
$$
T_{\rm OOCF}(x) = R^{k+1}(x) = \pi_1(\mathcal{R}^{\ell}(x,y)) = E(x).
$$
\begin{figure}[h]
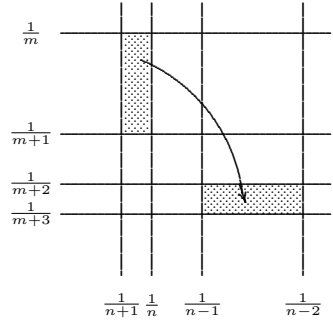

$$
\beginpicture
\setcoordinatesystem units <0.4cm,0.4cm>
    \setplotarea x from 0 to 11, y from 4 to 11
  \put {\tiny{$\tfrac{1}{n-2}$}} at 10 1
  \put {\tiny{$\tfrac{1}{n-1}$}} at 6.66666 1
  \put {\tiny{$\tfrac{1}{n}$}} at 5 1
  \put {\tiny{$\tfrac{1}{n+1}$}} at 4 1
  \put {\tiny{$\tfrac{1}{m}$}} at 1 10
  \put {\tiny{$\tfrac{1}{m+1}$}} at 1 6.66666
  \put {\tiny{$\tfrac{1}{m+2}$}} at 1 5
  \put {\tiny{$\tfrac{1}{m+3}$}} at 1 4
\circulararc 60 degrees from 8 4.7 center at 2.5 4
\arrow <4pt> [.2,.8] from 7.9 5 to 8.1 4.4

\setlinear \plot
2 10 11 10 
/
\setlinear \plot
10 2 10 11 
/
\setlinear \plot
2 6.66666 11 6.66666 
/
\setlinear \plot
6.66666 2 6.66666 11 
/
\setlinear \plot
2 5 11 5 
/
\setlinear \plot
5 2 5 11 
/
\setlinear \plot
2 4 11 4 
/
\setlinear \plot
4 2 4 11 
/
\setshadegrid span <1.15pt>
\vshade 4 6.66666 10
5 6.66666 10
/
\setshadegrid span <1.15pt>
\vshade 6.66666 4 5
10 4 5
/
\endpicture
$$ \caption[naturalextension]{The set $[\tfrac{1}{n+1},\tfrac{1}{n}]\times [\tfrac{1}{m+1},\tfrac{1}{m}]$ in $\mathcal{O}$, and its image under $\mathcal{R}$. Here $n\geq 3$} \label{case(i)}
\end{figure}

\item[($ii$)] Let $\tfrac{1}{3}\leq x\leq \tfrac{1}{2}$. Then from~\eqref{naturalextensionmapR2} it is immediately clear that $\mathcal{R}([\tfrac{1}{3},\tfrac{1}{2}]\times [0,1]) = [0,1]\times [\tfrac{1}{3},\tfrac{1}{2}]$, and again $k=1=\ell$, and we have the same conclusion as in case ($i$).

\item[ ]

\item[($iii$)] Let $\tfrac{1}{2}\leq x< 1$. Then there exist $n,m\in\N$, $n\geq 1$, and $m\geq 2$, such that 
$$
\tfrac{n}{n+1}\leq x < \tfrac{n+1}{n+2},\quad \text{and}\,\, \tfrac{1}{m+1}\leq y\leq \tfrac{1}{m}.
$$
Since $\mathcal{R}([\tfrac{1}{2},1]\times [0,1])= [0,1]\times [\tfrac{1}{2},1]$ (see also Figure~\ref{figurenatextandimageunderR}), it follows that $\ell\geq 1$. Also, in this case $x\not\in [0,\tfrac{1}{2})$, and it follows that $k\geq 1$. From~\eqref{Rimageinterval} and~\eqref{naturalextensionmapR2}  we have
$$
\mathcal{R}([\tfrac{n}{n+1},\tfrac{n+1}{n+2}]\times [\tfrac{1}{m+1},\tfrac{1}{m}]) = [\tfrac{n-1}{n},\tfrac{n}{n+1}]\times [\tfrac{m+1}{2m+1},\tfrac{m}{2m-1}],
$$
and in general, for $n\geq 1$, $m\geq 0$,
$$
\phantom{XXXX}\mathcal{R}([\tfrac{n-1}{n},\tfrac{n}{n+1}]\times [\tfrac{m}{m+1},\tfrac{m+1}{m+2}] = [\tfrac{n-1}{n},\tfrac{n}{n+1}]\times [\tfrac{m+1}{m+2},\tfrac{m+2}{m+3}]\subset [0,1]\times [\tfrac{1}{2},1].
$$
So under $\mathcal{R}$ we see that the points $(x,y)\in \mathcal{O}$ with $\tfrac{1}{2}\leq x<1$ ``moves first up and to the left,'' with its second coordinate in $[\tfrac{1}{2},\tfrac{2}{3}]$, and then (when $\pi_1(R(x,y))\in [\tfrac{1}{2},1]$) moves ``further up,'' as
$$
\frac{1}{2-\tfrac{m}{m+1}} = \frac{m+1}{m+2},\quad \text{for $m\in\N$}.
$$
Note that the first coordinate moves to the left after each application of $\mathcal{R}$. After finite many steps the orbit of $(x,y)$ under $\mathcal{R}$ is in $[0,\tfrac{1}{2}]\times [\tfrac{1}{2},1]$, and this number of steps is by definition $k$. Applying $\mathcal{R}$ one more time ``brings us back to $\mathcal{O}$ for the first time after we left it.'' So $k+1=\ell$.\smallskip\
\end{itemize}
These three cases prove the \textbf{claim}.\medskip\

The density\footnote{For ease of terminology we use here and in the rest of the paper the word \emph{density} loosely, as the measure is infinite. It would have been more appropriate to use the expression Radon-Nikodym derivative.} of the $T_{\rm OOCF}$-invariant measure is the projection of $\frac{1}{(x+y-2xy)^2}$ on the first coordinate;
\begin{eqnarray*}
\int_0^{\tfrac{1}{2}} \frac{1}{(x+y-2xy)^2}\, {\rm d}y &=& \quad \text{(set $u=x+y-2xy$, then ${\rm d}u=(1-2x)\, {\rm d}y$)}\\
&=& \int_x^{\tfrac{1}{2}} \frac{1}{u^2}\cdot \frac{1}{1-2x}\, {\rm d}u\,\, =\,\, \frac{1}{1-2x} \left[ \frac{-1}{u}\right]_x^{\tfrac{1}{2}}\\
&=& \frac{1}{1-2x}\left[ -2+\frac{1}{x}\right]\,\, =\,\, \frac{1}{x}.
\end{eqnarray*}
}\hfill $\triangle$
\end{remark}

\section{A strange continued fraction related to the Romik map}\label{sec:RomikMapCF}
Note that we can rewrite the Romik map $R$ as follows
\begin{equation}\label{RomikMapRrewrite}
R(x) = \begin{cases}
\displaystyle{\left( \frac{1}{x}-2\right)^{-1}}, & \text{for $0\leq x\leq \frac{1}{3}$}\medskip\ ;\\
\phantom{X}\displaystyle{\frac{1}{x}-2}, & \text{for $\tfrac{1}{3}\leq x\leq\tfrac{1}{2}$}\medskip\ ;\\
\displaystyle{-\left( \frac{1}{x}-2\right)}, & \text{for $\tfrac{1}{2}\leq x\leq 1$}.
\end{cases}
\end{equation}
So if we define 
$$
\varepsilon (x) = \begin{cases}
+1, & \text{if $0\leq x<\tfrac{1}{2}$}; \medskip\ \\
-1, & \text{if $\tfrac{1}{2}\leq x\leq 1$},
\end{cases}\quad \text{and}\quad
\delta (x) = \begin{cases}
-1, & \text{if $0\leq x<\tfrac{1}{3}$}; \medskip\ \\
+1, & \text{if $\tfrac{1}{3}\leq x\leq 1$},
\end{cases}
$$
and for $x\in [0,1]$ and $n\in\N$
$$
\varepsilon_n=\varepsilon_n(x):=\varepsilon (R^{n-1}(x)),\quad \text{and}\quad \delta_n=\delta_n(x):=\delta (R^{n-1}(x)),
$$
it follows immediately from~\eqref{RomikMapRrewrite} that for $x\in [0,1]$,
$$
R(x) = \varepsilon_1(x)\left( \frac{1}{x}-2\right)^{\delta_1(x)},
$$
from which we immediately find that\footnote{Here and in the rest of this paper we suppress the dependence of $\varepsilon_n$ and $\delta_n$ of $x$.}
\begin{equation}\label{RomikMapOnceUsed}
x = \frac{1}{2+\varepsilon_1 R(x)^{\delta_1}}.
\end{equation}
Similarly, we find that
\begin{equation}\label{RomikMapOnceUsedagain}
R(x) = \frac{1}{2+\varepsilon_2 (R^2(x))^{\delta_2}},\quad R^2(x) = \frac{1}{2+\varepsilon_3 (R^3(x))^{\delta_3}}
\end{equation}
and substituting these two equalities in~\eqref{RomikMapOnceUsed} yields
$$
x = \frac{1}{2+\displaystyle{\frac{\varepsilon_1}{\Big( 2+\varepsilon_2 (R^2(x))^{\delta_2}\Big)^{\delta_1}}}} 
 = \frac{1}{2+\displaystyle{\frac{\varepsilon_1}{\Big( 2+ \displaystyle \frac{\varepsilon_2}{(2+\varepsilon_3 (R^3(x))^{\delta_3})^{\delta_2}}\Big)^{\delta_1}}}}.
$$
Repeating this process $n$-times we would find inductively a complicated continued fraction, where some parts might collapse whenever $\delta_i=-1$ for $1\leq i\leq n$. It is insightful to see how this collapsing process works step by step using truncations (as one would do in case of the RCF). We first consider an example, where we use~\eqref{RomikMapOnceUsed} explicitly:
\begin{equation}\label{RomikMapOnceUsedvariant2}
x = \frac{1}{2+\varepsilon_1 R(x)^{\delta_1}}= \begin{cases}
\frac{1}{2+\frac{1}{R(x)}}, & \text{if $x\in [0,\tfrac{1}{3})$};\medskip\ \\
\frac{1}{2+R(x)}, & \text{if $x\in [\tfrac{1}{3},\tfrac{1}{2})$};\medskip\ \\
\frac{1}{2-R(x)}, & \text{if $x\in [\tfrac{1}{2},1]$}.
\end{cases}
\end{equation}

\begin{example}\label{example1}{\rm Let $x\in [\tfrac{8}{51},\tfrac{19}{121}]=[0.1568627\dots,0.157024\dots]$. For example, $x=g/3.938=0.156941084\dots$ from Example~\ref{ex:periodicexpansion}. In fact, the interval $[\tfrac{8}{51},\tfrac{19}{121}]$ is the cylinder of all $x$ with\footnote{These values are easily determined by using the beginning of the $R$-orbit of $x$ expressed as RCF-expansions; see also Table~\ref{table2} in Example~\ref{ex:periodicexpansion}.}
$$
\begin{array}{llll}
(\delta_1,\varepsilon_1)=(-1,1), & (\delta_2,\varepsilon_2)=(-1,1), & (\delta_3,\varepsilon_3)=(1,1), & (\delta_4,\varepsilon_4)=(1,1), \\
(\delta_5,\varepsilon_5)=(1,-1), & (\delta_6,\varepsilon_6)=(1,-1), & (\delta_7,\varepsilon_7)=(-1,1) & (\delta_8,\varepsilon_8)=(-1,1).
\end{array}
$$

In general, for $n\in\N$ and $(\delta_1,\dots,\delta_n,\varepsilon_1,\dots,\varepsilon_n)\in \{-1,1\}^{2n}$, with $(\delta_i,\varepsilon_i)\neq (-1,-1)$ for $1\leq i\leq n$, cylinders $\Delta_n=\Delta_n (\delta_1,\dots,\delta_n;\varepsilon_1,\dots,\varepsilon_n)$ are defined as
$$
\Delta_n (\delta_1,\dots,\delta_n;\varepsilon_1,\dots,\varepsilon_n) := \{ x\in [0,1]\, \big{|}\, \delta_i(x)=\delta_i, \varepsilon_i(x)=\varepsilon_i, \text{for $i=1,\dots,n$} \}.
$$
As all three branches of $R$ are full, we see that for all $n$ all cylinders are intervals of positive Lebesgue measure.\smallskip\

From~\eqref{RomikMapOnceUsedvariant2} we now see that for $x=g/3.938$,
$$
x = \frac{1}{2+\displaystyle \frac{1}{R(x)}}.
$$
In the RCF we would now set $R(x)$ equal to 0 in order to find the first convergent, which\footnote{Setting $1/0=\infty$ and $1/\infty = 0$.} would yield $0$ as the first convergent of $x$. Setting $R(x)$ equal to 1 we would find $1/3$ as first convergent of $x$. So either way we find as a possible first convergent of $x$ one of the end points of the cylinder $\Delta_1 (-1;1)=[0,\tfrac{1}{3}]$ since $R(x)\in [0,1]$.\smallskip\

Using the first equality in~\eqref{RomikMapOnceUsedagain} we find, since $\delta_2=-1$ and $\varepsilon_2=+1$,
$$
x = \frac{1}{2+\displaystyle \frac{1}{R(x)}} = \frac{1}{2+ \displaystyle \frac{1}{\displaystyle \frac{1}{2+\displaystyle \frac{1}{R^2(x)}}}} = \frac{1}{4+\displaystyle \frac{1}{R^2(x)}} .
$$
Setting $R^2(x)$ equal to 0 would again yield $0$ as the first convergent of $x$, while setting $R^2(x)$ equal to 1 would yield $\tfrac{1}{5}$ as the second convergent of $x$. Note that $R^2(x)\in [0,1]$ yields that $\Delta_2 (-1,-1;1,1)=[0,\tfrac{1}{5}]$.\smallskip\

Next, using the second equality in~\eqref{RomikMapOnceUsedagain} we find, since $\delta_3=+1=\varepsilon_3$,
$$
x = \frac{1}{4+\displaystyle \frac{1}{R^2(x)}} = \frac{1}{4+ \displaystyle \frac{1}{\displaystyle \frac{1}{2+ R^3(x)}}} = \frac{1}{6+R^3(x)}.
$$
So after ``collapsing'' we see that setting $R^3(x)$ equal to 0 we find  $\tfrac{1}{6}$ not as the third but as the first convergent of $x$, while setting $R^3(x)$ equal to 1 yields $\tfrac{1}{7}$ not as the third but as the first convergent; the ``collapsing'' has ``wiped out'' earlier convergents! It goes without saying that these two points $\tfrac{1}{7}$ resp.\ $\tfrac{1}{6}$ are the endpoints of $\Delta_3 (-1,-1,+1;+1,+1,+1)$.\smallskip\

The next three steps are more straightforward,
\begin{equation}\label{3steps}
x = \frac{1}{6+\displaystyle \frac{1}{2+R^4(x)}} = \frac{1}{6+\displaystyle \frac{1}{2+\displaystyle \frac{1}{2-R^5(x)}}} = \frac{1}{6+\displaystyle \frac{1}{2+\displaystyle \frac{1}{2-\displaystyle \frac{1}{2-R^6(x)}}}} 
\end{equation}
Setting $R^i(x)$ equal to 0 for $i=4,5,6$ in~\eqref{3steps} yields as convergents $\tfrac{1}{6}$ and $\tfrac{2}{13}$ (for $i=4$),  $\tfrac{1}{6}$, $\tfrac{2}{13}$, and $\frac{5}{32}$ (for $i=5$), $\tfrac{1}{6}$, $\tfrac{2}{13}$, $\frac{5}{32}$, and $\tfrac{8}{51}$ (for $i=6$).\smallskip\

As the last three steps were reminiscent of ``ordinary'' continued fraction algorithms, since $\delta_7=-1$ (and automatically $\varepsilon_7=+1$, as the pair $(\delta_i,\varepsilon_i)=(-1,-1)$ does not occur), the next step will involve again a ``collapse.'' From the last equation in~\eqref{3steps} we see, that
$$
x =  \frac{1}{6+\displaystyle \frac{1}{2+\displaystyle \frac{1}{2-\displaystyle \frac{1}{2-R^6(x)}}}} =  \frac{1}{6+\displaystyle \frac{1}{2+\displaystyle \frac{1}{2-\displaystyle \frac{1}{2-\displaystyle \frac{1}{2+\displaystyle \frac{1}{R^7(x)}}}}}}.
$$
}\hfill $\triangle$
\end{example}

The pattern is clear! Every time we have $\delta_n=-1$ (and hence $\varepsilon_n=+1$), we get a ``collapsible fraction'' of the form
\begin{equation}\label{collapsingfractions}
\frac{1}{2+\displaystyle \frac{1}{{\color{red}{0}}+R^n(x)}} = \frac{1}{2+\displaystyle \frac{1}{{\color{red}{0}}+\displaystyle\frac{1}{2+\varepsilon_{n+1}(R^{n+1}(x))^{\delta_{n+1}}}}} = \frac{1}{4+\varepsilon_{n+1}(R^{n+1}(x))^{\delta_{n+1}}}.
\end{equation}
If $\delta_{n+1}=\cdots=\delta_{n+m}=-1$, $\delta_{n+m+1}=+1$, for some $m\in\N\cup \{0\}$ (and therefore $\varepsilon_{n+1}=\cdots=\varepsilon_{n+m}=+1$), we can repeat this ``collapsing'' another $m$ times to arrive at (and note that $\delta_{n+m+1}=1$)
\begin{equation}\label{aftercollapsingfractions}
\frac{1}{2+\displaystyle \frac{1}{{\color{red}{0}}+R^n(x)}} = \frac{1}{2(m+2)+\varepsilon_{n+m+1}(R^{n+m+1}(x))}.
\end{equation}
Note that the partial quotient $2(m+2)$ in~\eqref{aftercollapsingfractions} will not change anymore, even if there are ``collapses'' after ``time'' $n+m+1$. We see that after ``collapsing'', we are left with a continued fraction expansion of $x$ with \textbf{only} even partial quotients. The above described ``collapsing'' is in fact the \emph{jump transformation} described by Kim \emph{et al}.\ in~\cite{[KLL]}. We also see from~\eqref{collapsingfractions} and~\eqref{aftercollapsingfractions}, where we inserted ${\color{red}{0}}$ as a ``missing'' partial quotient, that before collapsing every $x\in [0,1]$ has a formal continued fraction expansion of the form
\begin{equation}\label{NewCF1}
x = \frac{1}{a_1 +\displaystyle \frac{\rho_1}{a_2+\displaystyle\frac{\rho_2}{a_3+\ddots}}},
\end{equation}
with $\rho_n\in\{ -1,+1\}$ and $a_n\in \{ 0,2\}$, and where $a_n=0$ implies $\rho_n=+1$ (and also $\rho_{n-1}=+1$ and $a_{n-1}=2$). We denote the continued fraction in~\eqref{NewCF1} as
\begin{equation}\label{NewCF2}
x = [0; 1/a_1,\rho_1/a_2,\rho_2/a_3,\dots].
\end{equation}
We call the continued fraction expansion of $x$ in~\eqref{NewCF1} and~\eqref{NewCF2} with the above constraints the \emph{Romik} (\emph{continued fraction}) \emph{expansion} of $x$.

For example, if $x=g/3.938=0.156941084\dots$ (see Example~\ref{example1}), then
$$
x = [0; 1/2, 1/0, 1/2, 1/0, 1/2, 1/2, 1/2, -1/2, -1/2, \dots ],
$$
which we ``abbreviate'' as
\begin{equation}\label{NewCFexample}
x = [0;(1/2, 1/0)^2, (1/2)^3, (-1/2)^2,  \dots ].
\end{equation}
In fact, a rewrite of~\eqref{RomikMapOnceUsedvariant2}  yields how for $n\in\N\cup\{ 0\}$ this ``strange'' continued fraction is ``built up'' (in fact~\eqref{RomikMapOnceUsedvariant2} is the case $n=0$ of~\eqref{RomikMapOnceUsedvariant3}):
\begin{equation}\label{RomikMapOnceUsedvariant3}
R^n(x) = \begin{cases}
\frac{1}{2+\frac{1}{{\color{red}{0}} + R^{n+1}(x)}}, & \text{if $R^n(x)\in [0,\tfrac{1}{3})$; i.e., $\delta_{n+1}=-1, \varepsilon_{n+1}=1$};\medskip\ \\
\frac{1}{2+R^{n+1}(x)}, & \text{if $R^n(x)\in [\tfrac{1}{3},\tfrac{1}{2})$; i.e., $\delta_{n+1}=1, \varepsilon_{n+1}=1$};\medskip\ \\
\frac{1}{2-R^{n+1}(x)}, & \text{if $R^n(x)\in [\tfrac{1}{2},1]$; i.e., $\delta_{n+1}=1, \varepsilon_{n+1}=-1$}.
\end{cases}
\end{equation}\medskip\

We now formalize Example~\ref{example1} and the remarks following it in a few theorems.

\begin{theorem}\label{thm:RomikCF}
Let $x\in (0,1)$, let $n\in\N$, and set $m=\sum_{j=0}^{n-1} 1_{[0,\tfrac{1}{3})}\left( R^j(x)\right)$. Then there exist $a_1,\dots , a_{n+m}\in \{ 0,2\}$ and $\rho_1,\dots,\rho_{n+m}\in \{ -1,+1\}$, such that $a_1=2$, and $a_i=0$ implies $a_{i-1}=2$ and $\rho_{i}=+1$, for $2\leq i\leq n+m$, and
\begin{equation}\label{RomikCF}
x = \frac{1}{a_1 +\displaystyle\frac{\rho_1}{a_2+\ddots +\displaystyle\frac{\rho_{n+m-1}}{a_{n+m} + \rho_{n+m}R^n(x)}}},
\end{equation}
\end{theorem}

\begin{proof}
The statement that $a_1=2$ follows from~\eqref{RomikMapOnceUsedvariant2}, and that $a_i=0$ implies $a_{i-1}=2$ and $\rho_i=+1$, for $2\leq i\leq n+m$, is an immediate consequence of the first line in~\eqref{RomikMapOnceUsedvariant3}. For statement~\eqref{RomikCF} we give a proof by mathematical induction.
\begin{itemize}
\item[($i$)] The case $n=1$. In this case we ave three possible sub-cases.
\begin{itemize}
\item[($j$)] $x\in(0,\tfrac{1}{3})$. According to~\eqref{RomikMapOnceUsedvariant2},
$$
x = \frac{1}{2+\displaystyle\frac{1}{0+R(x)}}.
$$
We see that $m=1$, $a_1=2$, $a_{n+m}=a_2=0$, $\rho_1=+1$, $\rho_2=+1$ (since $R(x)\geq 0$), and~\eqref{RomikCF} holds for $n=1$ in this case.\smallskip\

\item[($jj$)] $x\in [\tfrac{1}{3},\tfrac{1}{2})$. According to~\eqref{RomikMapOnceUsedvariant2},
$$
x = \frac{1}{2+R(x)}.
$$
We see that $m=0$, $a_1=2$, $\rho_1=+1$ (since $R(x)\geq 0$), and~\eqref{RomikCF} holds for $n=1$ in this case.\smallskip\

\item[($jjj$)] $x\in [\tfrac{1}{2},1)$. This case is similar to the previous case ($jj$), the only difference being that now $\rho_1=-1$ (since $R(x)\geq 0$). We again see that~\eqref{RomikCF} holds for $n=1$ in this case.\medskip\
\end{itemize}

\item[($ii$)] \emph{Induction hypothesis}: Assume~\eqref{RomikCF} holds for some $n\in\N$. We must show it holds for $n+1$. Again we have three possible sub cases.\smallskip\

\begin{itemize}
\item[($\ell$)] $R^n(x)\in(0,\tfrac{1}{3})$. According to~\eqref{RomikMapOnceUsedvariant2}, where $R^n(x)$ is replaced by
$$
\frac{1}{2+\displaystyle\frac{1}{0 + R^{n+1}(x)}}, 
$$
it follows that~\eqref{RomikCF} becomes
$$
x = \frac{1}{a_1 +\displaystyle\frac{\rho_1}{a_2+\ddots +\displaystyle\frac{\rho_{n+m-1}}{a_{n+m} + \displaystyle\frac{\rho_{n+m}}{2+\displaystyle\frac{1}{0 + R^{n+1}(x)}}}}}.
$$
Note that now
\begin{eqnarray*}
{\phantom{XXX}}\sum_{j=0}^n 1_{[0,\tfrac{1}{3})}(R^j(x)) &=& \sum_{j=0}^{n-1} 1_{[0,\tfrac{1}{3})}(R^j(x)) + 1_{[0,\tfrac{1}{3})}(R^n(x)) \\
&=& m+1,
\end{eqnarray*}
and setting 
$$
a_{n+m+1}=2,\, a_{n+m+2}=0,\, \rho_{n+m+1}=+1,
$$
and $\rho_{n+m+2}=+1$ (since $R^{n+1}(x)\geq 0$), we see that if~\eqref{RomikCF} holds for some $n\in\N$ and $R^n(x)\in [0,\tfrac{1}{3})$, \eqref{RomikCF} also holds for $n+1$.\smallskip\

\item[($\ell\ell$)] $R^n(x)\in [\tfrac{1}{3},\tfrac{1}{2})$. According to~\eqref{RomikMapOnceUsedvariant2}, where $R^n(x)$ is replaced by
$$
\frac{1}{2+R^{n+1}(x)}, 
$$
it follows that~\eqref{RomikCF} becomes
$$
x = \frac{1}{a_1 +\displaystyle\frac{\rho_1}{a_2+\ddots +\displaystyle\frac{\rho_{n+m-1}}{a_{n+m} + \displaystyle\frac{\rho_{n+m}}{2+ R^{n+1}(x)}}}}.
$$
Note that now
\begin{eqnarray*}
{\phantom{XXX}}\sum_{j=0}^n 1_{[0,\tfrac{1}{3})}(R^j(x)) &=& \sum_{j=0}^{n-1} 1_{[0,\tfrac{1}{3})}(R^j(x)) + 1_{[0,\tfrac{1}{3})}(R^n(x)) \\
&=& m+0,
\end{eqnarray*}
and setting 
$$
a_{n+m+1}=2, \rho_{n+m+1}=+1\,\, \text{(since $R^{n+1}(x)\geq 0$)},
$$
we see that if~\eqref{RomikCF} holds for some $n\in\N$ and $R^n(x)\in [\tfrac{1}{3},\tfrac{1}{2})$, \eqref{RomikCF} also holds for $n+1$.\smallskip\

\item[($\ell\ell\ell$)] $R^n(x)\in [\tfrac{1}{2},1]$. This case is similar to the previous case ($\ell\ell$), the only difference being that now $\rho_{n+m+1}=-1$ (since $R^{n+1}(x)\geq 0$). We see that if~\eqref{RomikCF} holds for some $n\in\N$ and $R^n(x)\in [\tfrac{1}{2},1]$, \eqref{RomikCF} also holds for $n+1$.\
\end{itemize}
\end{itemize}
From ($i$) and ($ii$) it now follows that~\eqref{RomikCF} holds for all $n\in\N$.
\end{proof}

\begin{example}\label{example2}{\rm  (continuation of Example~\ref{example1}). Again, let $x=g/3.938=0.156941084\dots$. Taking finite truncations in~\eqref{NewCFexample} yields as Romik-convergents (the convergents in {\color{blue}blue} are also RCF-convergents of $x$ (from Table~\ref{table1}), the convergents in {\color{red}red} occurred earlier),
$$
\begin{array}{ll}
\frac{p_1}{q_1}=[0; 1/2] = \tfrac{1}{2}; & {\color{blue}\frac{p_2}{q_2}} = [0;1/2, 1/0] = {\color{blue}0}; \\
\frac{p_3}{q_3}=[0; 1/2, 1/0,1/2]=\tfrac{1}{4}; & {\color{red}\frac{p_4}{q_4}} = [0; 1/2, 1/0,1/2,1/0]={\color{red}0};\\
{\color{blue}\frac{p_5}{q_5}} = [0; 1/2, 1/0,1/2,1/0,1/2]={\color{blue}\tfrac{1}{6}}; & {\color{blue}\frac{p_6}{q_6}} = [0; 1/2, 1/0,1/2,1/0,1/2,1/2]={\color{blue}\tfrac{2}{13}}.
\end{array}
$$
As the notation is a bit cumbersome, we write $\tfrac{p_6}{q_6}$ as
$$
\tfrac{p_6}{q_6}=[0; (1/2, 1/0)^2,(1/2)^2].
$$
With this notation, which we used for $x$ in~\eqref{NewCFexample},
$$
\tfrac{p_7}{q_7}=[0; (1/2, 1/0)^2,(1/2)^3]=\tfrac{5}{32};\quad {\color{blue}\tfrac{p_8}{q_8}}\, = [0; (1/2, 1/0)^2,(1/2)^3,-1/2] = {\color{blue}\tfrac{8}{51}}; 
$$
and $\tfrac{p_9}{q_9}=[0; (1/2, 1/0)^2,(1/2)^3,(-1/2)^2] = \tfrac{11}{70}$.\medskip\

If we compare these Romik-convergents with the RCF-convergent of $x$ in Example~\ref{ex:periodicexpansion}, we see that the third RCF-convergent $\tfrac{P_3}{Q_3}=\tfrac{3}{19}$ from Table~\ref{table1} is \textbf{missing}. From this example it is already immediately clear that not all RCF-convergents are Romik-convergents; something which is certainly the case of the Farey-map.\hfill $\triangle$
}
\end{example}

\section{Convergents of the Romik map}\label{convergentsNewCF}
Taking finite truncations in~\eqref{NewCF1} or~\eqref{NewCF2} yields a sequence of rational numbers $(p_n/q_n)_{n\geq 1}$, where $p_n,q_n\in\Z$, $q_n>0$, $\text{gcd}\{ p_n,q_n\}=1$, and
\begin{equation}\label{NewCF3}
\frac{p_n}{q_n} = \frac{1}{a_1 +\displaystyle \frac{\rho_1}{a_2+ \ddots +\displaystyle\frac{\rho_{n-1}}{a_n}}},
\end{equation}
which we abbreviate as
\begin{equation}\label{NewCF4}
\frac{p_n}{q_n} = [0; 1/a_1,\rho_1/a_2,\dots,\rho_{n-1}/a_n].
\end{equation}
This sequence $(p_n/q_n)_{n\geq 1}$ is finite if $x\in [0,1]$ is a rational number and there exists an $n\in\N$ such that $R^n(x)=0$, and is infinite otherwise; see also Proposition~\ref{prop:RationalsHaveFiniteOrbit}.\smallskip\

Setting $\rho_0:=1$, we now define matrices\footnote{With the obvious restriction on $n$ if the formal continued fraction expansion~\eqref{NewCF2} of $x$ is finite.} $A_n=A_n(x)$ and $M_n=M_n(x)$ as
\begin{equation}\label{AnMnmMatrices}
A_n=A_n(x) = \begin{pmatrix}
0 & \rho_{n-1}\\
1 & a_n
\end{pmatrix},\quad
M_n=M_n(x)=A_1A_2\cdots A_n,\quad \text{for $n\in\N$}.
\end{equation}
The matrices $A_n$ and $M_n$ have integer entries and determinants $\pm 1$, so elements of the modular group $\text{SL}(2,\Z)$. Now every element $A$ of the modular group can be viewed as a \emph{M\"obius} (or: \emph{fractional linear}) transformation; if
$$
A = \begin{pmatrix}
a & b\\
c & d
\end{pmatrix},\quad \text{with $a,b,c,d\in\Z$ and $ad-bc=\pm 1$},
$$
then for $x\in\R$ we define the map $x\mapsto A\cdot x\in\R\cup\{\pm\infty\}$ as
$$
A\cdot x = \begin{pmatrix}
a & b\\
c & d
\end{pmatrix}\cdot x := \frac{ax+b}{cx+d}.
$$
It is an easy exercise to see, that if $A,B\in \text{SL}(2,\Z)$, we have for $x\in\R$,
$$
(AB)\cdot x = A\cdot (B\cdot x).
$$
\begin{remark}\label{remarkonmatrices}{\rm 
The Romik map $R$ from~\eqref{RomikMapR} can be written as a M\"obius transformation, 
$$
R(x) = \begin{cases}
\begin{pmatrix}
1 & 0\\
-2 & 1
\end{pmatrix}\cdot x, & \text{if $x\in [0,\tfrac{1}{3}]$}\\
\begin{pmatrix}
-2 & 1\\
1 & 0
\end{pmatrix}\cdot x, & \text{if $x\in [\tfrac{1}{3},\tfrac{1}{2}]$}\\
\begin{pmatrix}
2 & -1\\
1 & 0
\end{pmatrix}\cdot x, & \text{if $x\in [\tfrac{1}{2},1]$},  
\end{cases}
$$
and consequently we have, if we multiply left \& right with the inverses of these matrices,
$$
x = \begin{cases}
\begin{pmatrix}
1 & 0\\
2 & 1
\end{pmatrix}\cdot R(x), & \text{if $x\in [0,\tfrac{1}{3}]$}\\
\begin{pmatrix}
0 & 1\\
1 & 2
\end{pmatrix}\cdot R(x), & \text{if $x\in [\tfrac{1}{3},\tfrac{1}{2}]$}\\
\begin{pmatrix}
0 & 1\\
-1 & 2
\end{pmatrix}\cdot R(x), & \text{if $x\in [\tfrac{1}{2},1]$},  
\end{cases}
$$
For $x\in [0,\tfrac{1}{3}]$ we have $a_1=2$, $a_2=0$, and indeed
$$
M_2=A_1A_2 = \begin{pmatrix}
0 & 1\\
1 & 2   
\end{pmatrix}
\begin{pmatrix}
0 & 1\\
1 & 0
\end{pmatrix} = \begin{pmatrix}
1 & 0\\
2 & 1
\end{pmatrix},
$$
If $x\in [\tfrac{1}{3},\tfrac{1}{2}]$ we have $a_1=2$, and 
$$
M_1=A_1= \begin{pmatrix}
0 & 1\\
1 & 2   
\end{pmatrix}.
$$
In case $x\in [\tfrac{1}{2},1]$, so $a_1=1$ and $\rho_1=-1$, the inverse of the matrix of the Romik maps is slightly different from $A_1$. This is due to the way we truncate the formal Romik continued fraction expansion; the $-1$ of $\rho_1$ is ``given'' to the next matrix $A_2$; see also definition~\eqref{AnMnmMatrices} of the matrix $A_n$. As we must have in this case that $a_2=2$, we see that
$$
\begin{pmatrix}
0 & 1\\
-1 & 2
\end{pmatrix}\cdot 0 = \tfrac{1}{2} = \begin{pmatrix}
0 & 1\\
1 & 2
\end{pmatrix}\cdot 0,
$$
\mbox{ and }
$$
\begin{pmatrix}
0 & 1\\
-1 & 2
\end{pmatrix} \begin{pmatrix}
0 & 1\\
1 & 2
\end{pmatrix} = \begin{pmatrix}
0 & 1\\
1 & 2
\end{pmatrix} \begin{pmatrix}
0 & -1\\
1 & 2
\end{pmatrix}.
$$
The overall effect remains the same; we obtain convergents $\tfrac{p_1}{q_1}=\tfrac{1}{2}$, $\tfrac{p_2}{q_2}=\tfrac{2}{3}$.\hfill$\triangle$
}
\end{remark}

Using induction, the following result is almost immediate, see also~\cite{[DK1],[IK]}.

\begin{lemma}\label{lemma:matrices1}
Let $x\in [0,1]$, with formal Romik continued fraction expansion~\eqref{NewCF1}=\eqref{NewCF2}, and convergents $\tfrac{p_n}{q_n}$, for $n\geq 1$ given by~\eqref{NewCF3}=\eqref{NewCF4}. Then we have, again with the obvious restriction on $n$ if the formal expansion of $x$ is finite,
\begin{itemize}
\item[($i$)] For $n\geq 1$,
$$
\frac{p_n}{q_n} = M_n\cdot 0.
$$

\item[($ii$)] For $n\geq 1$, and setting $p_0=p_0(x):=0$, $q_0=q_0(x):=1$,
$$
M_n = \begin{pmatrix}
p_{n-1} & p_n\\
q_{n-1} & q_n    
\end{pmatrix}.
$$
Furthermore,
$$
\text{\rm det}(M_n)=p_{n-1}q_n-p_nq_{n-1}=(-1)^{n-1}\rho_1\cdots\rho_{n-1}\in\{ \pm1\}, 
$$
and from this it follows that
$$
\text{\rm gcd}\{ p_n,q_n\} = \pm 1,\quad \text{for $n\geq 0$}.
$$

\item[($iii$)] For $n\geq 1$, and setting $p_{-1}=p_{-1}(x):=1$, $q_{-1}=q_{-1}(x):=0$, and
$$
M_0 = M_0(x) := \begin{pmatrix}
1 & 0\\
0 & 1    
\end{pmatrix},
$$
we have the recurrence relations
\begin{equation}\label{recurrencerelations}
\begin{array}{ccc}
p_{-1}=1; & p_0=0; & p_{n+1}=a_{n+1}p_n+\rho_np_{n-1},\quad n\geq 0;\\
q_{-1}=0; & q_0=1 & q_{n+1}=a_{n+1}q_n+\rho_nq_{n-1},\quad n\geq 0;
\end{array}
\end{equation}
\item[ ]
\end{itemize}
\end{lemma}

\noindent As the \emph{proof} of this lemma is fairly straightforward and can be found in several papers (like~\cite{[K]}, p.~3, 4) and books~(\cite{[DK1],[IK]}), it is skipped.\smallskip\

Although Lemma~\ref{lemma:matrices1} is quite standard, due to the special nature of the Romik map $R$ remarkable things happen, which we already saw in Example~\ref{example2}: some Romik convergents are repeated! We have the following result.

\begin{proposition}\label{prop:repeatedconvergents}
Let $x\in [0,1]$, with formal Romik continued fraction expansion~\eqref{NewCF1}=\eqref{NewCF2} and convergents $\tfrac{p_n}{q_n}$, for $n\geq 1$ given by~\eqref{NewCF3}=\eqref{NewCF4}. Then 
$$
\text{$a_n=0$ implies that $\tfrac{p_n}{q_n}=\tfrac{p_{n-2}}{q_{n-2}}$ and\, $\tfrac{p_{n+1}}{q_{n+1}}=\tfrac{2p_{n-2}+p_{n-1}}{2q_{n-2}+q_{n-1}}$.}
$$
\end{proposition}

\begin{proof}
Although the proof is very easy, we give it for completeness. As $a_n=0$,~\eqref{RomikMapOnceUsedvariant2} yields that $\rho_{n-1}=+1$, and
$$
M_n=\begin{pmatrix}
p_{n-1} & p_n\\
q_{n-1} & q_n    
\end{pmatrix}
= \begin{pmatrix}
p_{n-2} & p_{n-1}\\
q_{n-2} & q_{n-1}
\end{pmatrix}
\begin{pmatrix}
0 & 1 \\
1 & 0
\end{pmatrix} 
=
\begin{pmatrix}
p_{n-1} & p_{n-2}\\
q_{n-1} & q_{n-2}
\end{pmatrix}.
$$
If $a_n=0$ we see from~\eqref{RomikMapOnceUsedvariant2} that $a_{n+1}=2$ and $\rho_n=+1$. So,
\begin{eqnarray*}
M_{n+1} &=& \begin{pmatrix}
p_n & p_{n+1}\\
q_n & q_{n+1}    
\end{pmatrix}
= \begin{pmatrix}
p_{n-2} & p_{n-1}\\
q_{n-2} & q_{n-1}
\end{pmatrix}
\begin{pmatrix}
0 & 1\\
1 & 0
\end{pmatrix}
\begin{pmatrix}
0 & 1\\
1 & 2
\end{pmatrix}\\
&=&
\begin{pmatrix}
p_{n-1} & p_{n-2}\\
q_{n-1} & q_{n-2}
\end{pmatrix}\begin{pmatrix}
0 & 1\\
1 & 2
\end{pmatrix}\\
&=&
\begin{pmatrix}
p_{n-2} & 2p_{n-2}+p_{n-1}\\
q_{n-2} & 2q_{n-2}+q_{n-1}
\end{pmatrix}.
\end{eqnarray*}
\end{proof}

\begin{example}\label{example3}{\rm  (continuation of Examples~\ref{ex:periodicexpansion},~\ref{example1} and~\ref{example2}). Let $x$ be as in Example~\ref{ex:periodicexpansion}, and recall that in~\eqref{NewCFexample} we saw that
$$
x = [0;(1/2, 1/0)^2, (1/2)^3, (-1/2)^2,  \dots ].
$$
Using the formal Romik-expansion~\eqref{NewCFexample} of $x$ yields the first four $M_n$-matrices:
$$
\begin{array}{ll}
M_1=A_1 = \begin{pmatrix}
0 & 1 \\
1 & 2
\end{pmatrix},
&
M_2 = \begin{pmatrix}
0 & 1 \\
1 & 2
\end{pmatrix} 
\begin{pmatrix}
0 & 1\\
1 & 0
\end{pmatrix}
= \begin{pmatrix}
1 & 0\\
2 & 1
\end{pmatrix},\\
M_3 = \begin{pmatrix}
1 & 0 \\
2 & 1
\end{pmatrix} 
\begin{pmatrix}
0 & 1\\
1 & 2
\end{pmatrix}
= \begin{pmatrix}
0 & 1\\
1 & 4
\end{pmatrix},
&
M_4 = \begin{pmatrix}
0 & 1 \\
1 & 4
\end{pmatrix} 
\begin{pmatrix}
0 & 1\\
1 & 0
\end{pmatrix}
= \begin{pmatrix}
1 & 0\\
4 & 1
\end{pmatrix},
\end{array}
$$
and in the same way the next five $M_n$-matrices of $x$ are:
$$
\begin{array}{lll}
M_5 = \begin{pmatrix}
0 & 1\\
1 & 6
\end{pmatrix},
&
M_6 = \begin{pmatrix}
1 & 2\\
6 & 13
\end{pmatrix},
&
M_7 = \begin{pmatrix}
2 & 5\\
13 & 32
\end{pmatrix},\\
M_8 = \begin{pmatrix}
5 & 8\\
32 & 51
\end{pmatrix},
&
M_9 = \begin{pmatrix}
8 & 11\\
51 & 70
\end{pmatrix},
&

\end{array}
$$
We find as the first 9 convergents:
$$
\begin{array}{lllll}
\tfrac{p_0}{q_0}= 0, &\tfrac{p_1}{q_1}=\tfrac{1}{2}, & \tfrac{p_2}{q_2}=0, & \tfrac{p_3}{q_3}=\tfrac{1}{4}, &  \tfrac{p_4}{q_4}=0, \\
\tfrac{p_5}{q_5}=\tfrac{1}{6}, & \tfrac{p_6}{q_6}=\tfrac{2}{13}, & \tfrac{p_7}{q_7}=\tfrac{5}{32}, & \tfrac{p_8}{q_8}=\tfrac{8}{51}, & \tfrac{p_9}{q_9}=\tfrac{11}{70}.
\end{array}
$$
see also Example~\ref{example2}.\hfill $\triangle$
}
\end{example}

\subsection{Convergence of the Romik continued fraction expansion}\label{subsec:convergence}
For the RCF and ECF (and many other continued fraction algorithms like Nakada's $\alpha$-expansions (\cite{[N]}), Bosma's Optimal Continued Fraction expansion~(\cite{[Bosma]}; see also~\cite{[K]}), results similar to the following results are quite straightforward and the basis of not only convergence of the continued fraction algorithm, but are also used to find strong Diophantine properties. For the Romik continued fraction, due to the special nature of the Romik map in $[0,\tfrac{1}{3}]$, a much weaker result is obtained.

\begin{proposition}\label{prop:xasfunctionofRn(x)}
Let $x\in [0,1]$, with formal Romik continued fraction expansion~\eqref{NewCF1} {\rm (}which is denoted by~\eqref{NewCF2}{\rm )}, where we\footnote{In case~\eqref{NewCF1} is finite one can make easily the obvious changes.} assume~\eqref{NewCF1} is infinite. Let $n\in\N$, and let $m=m(n)$ be defined by
$$
m=m(n):= \# \{ 1\leq i\leq n\, \big{|}\, a_i=0\} .
$$
Then, if $k=n-m$,
\begin{itemize}
\item[($i$)]
\begin{equation}\label{xasfunctionofconvergents}
x = \begin{cases}
\frac{\displaystyle p_n+\rho_np_{n-1}R^k(x)}{\displaystyle q_n+\rho_nq_{n-1}R^k(x)}, & \text{if $a_{n+1}=2$};\\
    &  \\
\frac{\displaystyle p_{n-1} + p_nR^k(x)}{\displaystyle q_{n-1}+q_nR^k(x)}, & \text{if $a_{n+1}=0$}.
\end{cases}  
\end{equation}

\item[($ii$)]
\begin{equation}\label{|x-pn/qn|}
\left| x - \frac{p_n}{q_n}\right| = \begin{cases}
\frac{\displaystyle R^k(x)}{\displaystyle q_n\left( q_n+\rho_nq_{n-1}R^k(x)\right)}, & \text{if $a_{n+1}=2$};\\
   & \\
\frac{\displaystyle 1}{\displaystyle q_n\left( q_{n-1}+q_n R^k(x)\right)}, & \text{if $a_{n+1}=0$}.
\end{cases}
\end{equation}
\item[]
\end{itemize}
\end{proposition}

\begin{proof}
($i$) Let us first assume $a_{n+1}=2$. Then
$$
x = [0; 1/a_1,\rho_1/a_2,\dots, \rho_{n-1}/(a_n+\rho_nR^k(x))],
$$
and setting
$$
A_n^* = \begin{pmatrix}
0 & \rho_{n-1}\\
1 & a_n+\rho_nR^k(x)    
\end{pmatrix},
$$
one easily finds that
$$
x = M_{n-1}A_n^* \cdot 0.
$$

Now,
\begin{eqnarray*}
M_{n-1}\begin{pmatrix}
0 & \rho_{n-1}\\
1 & a_n+\rho_nR^k(x)    
\end{pmatrix} &=& \begin{pmatrix}
p_{n-2} & p_{n-1}\\
q_{n-2} & q_{n-1}
\end{pmatrix}\begin{pmatrix}
0 & \rho_{n-1}\\
1 & a_n+\rho_nR^k(x)    
\end{pmatrix} \\
&=& \begin{pmatrix}
p_{n-1} & a_np_{n-1} + \rho_{n-1}p_{n-2} + \rho_np_{n-1}R^k(x)\\
q_{n-1} & a_nq_{n-1} + \rho_{n-1}q_{n-2} + \rho_nq_{n-1}R^k(x)
\end{pmatrix}\\
&=& \begin{pmatrix}
p_{n-1} & p_n+\rho_{n}p_{n-1}R^k(x)\\
q_{n-1} & q_n+\rho_{n}q_{n-1}R^k(x)
\end{pmatrix};
\end{eqnarray*}
combining these last two results yields
$$
x = \frac{\displaystyle p_n+\rho_{n}p_{n-1}R^k(x)}{\displaystyle q_n+\rho_{n}q_{n-1}R^k(x)}.
$$\medskip

Next, assume that $a_{n+1}\neq 2$, i.e., $a_{n+1}=0$. Then
$$
x = [0; 1/a_1,\rho_1/a_2,\dots, \rho_{n-1}/a_n, 1/(0+R^k(x))],
$$
and setting
$$
A_n^{\sharp} = \begin{pmatrix}
0 & \rho_{n-1}\\
1 & a_n+\frac{1}{R^k(x)}    
\end{pmatrix},
$$
one finds
\begin{eqnarray*}
x &=& M_{n-1}A_n^{\sharp}\cdot 0\\
&=& \begin{pmatrix}
p_{n-2} & p_{n-1}\\
q_{n-2} & q_{n-1}
\end{pmatrix}\begin{pmatrix}
0 & \rho_{n-1}\\
1 & a_n+\frac{1}{R^k(x)}    
\end{pmatrix}\cdot 0\\
&=& \begin{pmatrix}
p_{n-1} & a_np_{n-1} + \rho_{n-1}p_{n-2} + \frac{p_{n-1}}{R^k(x)}\\
q_{n-1} & a_nq_{n-1} + \rho_{n-1}q_{n-2} + \frac{q_{n-1}}{R^k(x)}
\end{pmatrix}\cdot 0 \\
&=& \begin{pmatrix}
p_{n-1} & p_n+\frac{p_{n-1}}{R^k(x)}\\
q_{n-1} & q_n+\frac{q_{n-1}}{R^k(x)}
\end{pmatrix}\cdot 0\\
&=& \frac{\displaystyle p_{n-1}+ p_nR^k(x)}{\displaystyle q_{n-1}+q_nR^k(x)}.\\
& &
\end{eqnarray*}

($ii$) First assume that $a_{n+1}=2$. Then from ($i$) we find
\begin{eqnarray*}
\left| x - \frac{p_n}{q_n}\right| &=& \left| \frac{\displaystyle p_n+\rho_np_{n-1}R^k(x)}{\displaystyle q_n+\rho_nq_{n-1}R^k(x)} - \frac{p_n}{q_n}\right| \\
&=& \frac{R^k(x)}{q_n \big| q_n+\rho_nq_{n-1}R^k(x)\big|} .
\end{eqnarray*}
Now either $a_n=0$ or $a_n=2$. In the former case $\rho_n=+1$, so $q_n+\rho_nq_{n-1}R^k(x) > 0$. In the latter case $\rho_n=-1$ is possible, but since in this case $q_n>q_{n-1}$ and $R^k(x)\in [0,1]$, we again find that $q_n+\rho_nq_{n-1}R^k(x) > 0$.\medskip\

Next, let $a_{n+1}=0$, implying that $a_n=2$ and $\rho_n=+1$. Then from ($i$) we find
\begin{eqnarray*}
\left| x - \frac{p_n}{q_n}\right| &=& \left| \frac{\displaystyle p_{n-1}+ p_nR^k(x)}{\displaystyle q_{n-1}+q_nR^k(x)} - \frac{p_n}{q_n}\right| \\
&=& \frac{1}{q_n \big( q_{n-1}+q_{n}R^k(x)\big)} .
\end{eqnarray*}  
\end{proof}

In Example~\ref{example3} we saw, that for $x=g/3.938=0.156941084\dots$,
$$
\begin{array}{lllll}
q_0=1, & q_1=2, & q_2=1, & q_3=4, & q_4=1,\\
q_5=6, & q_6=13, & q_7=32, & q_8=51, & q_9=70
\end{array}
$$
So we see from this example that the sequence $(q_n)_{n\geq 0}$ is \emph{not} monotonically increasing (as is the case for many other continued fraction algorithms, such as the RCF). In fact, repetitions of convergents occur exactly when $R^i(x)\in (0,\tfrac{1}{3})$: in these cases, there is a stretch where every other convergent is the same, due to the following observation. Let $x\in (0,1)\setminus\Q$ be such that $a_n=2$ for some $n\in\N$ and $a_n=0$ when $n=0$, and 
$a_{n+2k+1}=2$, $a_{n+2k+2}=0$, for $0\leq k\leq N$, where $N\in\N$ and $a_{n+2N+4}=2$ (such an $N$ must exist as $x$ is irrational; we cannot have that $R^m(x)\in [0,\tfrac{1}{3})$ for all $m$ sufficiently large as this would imply that $x\in\Q$). Setting
$$
A := \begin{pmatrix}
0 & 1\\
1 & 2
\end{pmatrix} \begin{pmatrix}
0 & 1\\
1 & 0
\end{pmatrix} = \begin{pmatrix}
1 & 0\\
2 & 1
\end{pmatrix},
$$
we see using induction that
$$
A^k = \begin{pmatrix}
1 & 0\\
2k & 1
\end{pmatrix},\quad \text{for $k\in\N\cup\{ 0\}$},
$$
and from this it follows that
$$
A^k\begin{pmatrix}
0 & 1\\
1 & 2    
\end{pmatrix} = \begin{pmatrix}
0 & 1\\
1 & 2(k+1)
\end{pmatrix}.
$$
But then we find for $0\leq k\leq N$,
\begin{eqnarray*}
M_{n+2k+1} &=& M_nA^k\begin{pmatrix}
0 & 1\\
1 & 2    
\end{pmatrix} \,\, = \,\, \begin{pmatrix}
p_{n-1} & p_n\\
q_{n-1} & q_n
\end{pmatrix} \begin{pmatrix}
0 & 1\\
1 & 2(k+1)
\end{pmatrix}\\
&=& \begin{pmatrix}
p_n & 2(k+1)p_n+p_{n-1}\\
q_n & 2(k+1)q_n+q_{n-1}    
\end{pmatrix},
\end{eqnarray*}
and
\begin{eqnarray*}
M_{n+2k+2} &=& M_nA^{k+1} \,\, = \,\, \begin{pmatrix}
p_{n-1} & p_n\\
q_{n-1} & q_n
\end{pmatrix} \begin{pmatrix}
1 & 0\\
2(k+1) & 1
\end{pmatrix}\\
&=& \begin{pmatrix}
2(k+1)p_n+p_{n-1} & p_n\\
2(k+1)q_n+q_{n-1} & q_n   
\end{pmatrix}.
\end{eqnarray*}
Thus we see $N$-times the same convergent $\tfrac{p_n}{q_n}$, while meanwhile the denominators of the other convergents $\tfrac{2(k+1)p_n+p_{n-1}}{2(k+1)q_n+q_{n-1}}$, for $k=0,1,\dots,N$, are monotonically increasing. As we must have $a_{n+2N+3}=2$, and since we choose $N\in\N$ such that $a_{n+2N+4}=2$, we see that
$$
M_{n+2N+3} = \begin{pmatrix}
2(k+1)p_n+p_{n-1} & p_n\\
2(k+1)q_n+q_{n-1} & q_n   
\end{pmatrix} \begin{pmatrix}
0 & 1\\
1 & 2
\end{pmatrix} = \begin{pmatrix}
p_n & 2(k+2)p_n+p_{n-1}\\
q_n & 2(k+2)q_n+q_{n-1}
\end{pmatrix}
$$
and
\begin{eqnarray*}
M_{n+2N+4} &=& \begin{pmatrix}
p_n & 2(k+2)p_n+p_{n-1}\\
q_n & 2(k+2)q_n+q_{n-1}
\end{pmatrix}\begin{pmatrix}
0 & \rho_{n+2N+3}\\
1 & 2
\end{pmatrix} \\
&=& \begin{pmatrix}
2(k+2)p_n+p_{n-1} & (4k+9)p_n+(1+\rho_{n+2N+3})p_{n-1}\\
2(k+2)q_n+q_{n-1} & (4k+9)q_n+(1+\rho_{n+2N+3})q_{n-1}
\end{pmatrix}.
\end{eqnarray*}
Clearly, the following holds.
\begin{lemma}\label{growthofdenominatorsofqn}
Let $x\in [0,1]\setminus\Q$, with formal Romik continued fraction expansion~\eqref{NewCF1}=\eqref{NewCF2}, and convergents $\tfrac{p_n}{q_n}$, for $n\geq 1$ given by~\eqref{NewCF3}=\eqref{NewCF4}. Then for every bound $B>0$ there exists an index $N\in\N$, such that $q_n>B$ for all $n\geq N$.
\end{lemma}

\section{An algorithm to convert the RCF-expansion of any \emph{x} to its Romik-expansion (and a metric result)}\label{sec:AnConversionAlgorithm}
As we have already remarked several times, in~\cite{[KL]} an algorithm is given to convert the RCF-expansion of any $x\in\R$ to its ECF-expansion. This algorithm is based on \emph{insertions} and \emph{singularizations} (cf.~\eqref{insertion} resp.~\eqref{singularization}). Below we will give a similar algorithm to convert the RCF-expansion of any $x\in\R$ to the Romik-expansion of $x$.

\subsection{A conversion algorithm}\label{subsec:conversion}
Apart from the insertion from~\eqref{insertion}, and the singularization from~\eqref{singularization}, we will also need the following \emph{strange insertion}; let $a\in\N$, $a\geq 3$, and let $\xi\in [0,1)$, then
\begin{equation}\label{strangeinsertion}
\frac{1}{a+\xi} = \frac{1}{2+\displaystyle \frac{1}{0 + \displaystyle \frac{1}{a-2 + \xi}}}.
\end{equation}
In~\eqref{strangeinsertion} we have inserted 
$$
\frac{1}{2+\displaystyle \frac{1}{0\,\, + }}
$$
before the partial quotient $a$ (which becomes $a-2$). The effect of this insertion is the appearance of a mediant convergent and the repetition of an earlier convergent.

\begin{lemma}\label{lemma:strangeinsertion}
Let $x\in\R$, with continued faction expansion 
\begin{equation}\label{SomeExpansionOfX}
x=d_0+\frac{c_1}{d_1 +\displaystyle{\frac{c_2}{d_2+\ddots }}} = [d_0;c_1/d_1,c_2/d_2,\dots],
\end{equation}
with $d_0\in\Z$ such that $x-d_0\in [0,1)$, $d_n\in\N$ and $c_n\in\{ \pm 1\}$, for $n\in\N$, and with convergents $(\tfrac{r_n}{s_n})_{n\geq 0}$ converging to $x$. Furthermore, let for some $n\in\N$: $c_n=1$, $d_n\geq 3$. Then applying the ``strange insertion''~\eqref{strangeinsertion} to $c_n=1$, $d_n\geq 3$ yields a new continued fraction expansion of $x$ of the form
$$
x=[d_0;c_1/d_1,\dots, c_{n-1}/d_{n-1},1/2,1/0,1/(d_n-2),c_{n+1}/d_{n+1},\dots ],
$$
and where two new convergents {\rm (}viz.\ $\tfrac{2r_{n-1}+r_{n-2}}{2s_{n-1}+s_{n-1}}$ and $\tfrac{r_{n-1}}{s_{n-1}}${\rm )} are inserted between $\tfrac{r_{n-1}}{s_{n-1}}$ and $\tfrac{r_n}{s_n}$ {\rm (}so the convergent $\tfrac{r_{n-1}}{s_{n-1}}$ is repeated{\rm )}.
\end{lemma}

\begin{proof}
Similar to~\eqref{AnMnmMatrices}, define matrices $A_n=A_n(x)$ and $M_n=M_n(x)$ as
$$
A_0=A_0(x) = \begin{pmatrix}
1 & a_0\\
0 & 1
\end{pmatrix},\quad
A_n=A_n(x) = \begin{pmatrix}
0 & c_n\\
1 & d_n
\end{pmatrix},
$$
and $M_n=M_n(x) = A_0A_1A_2\cdots A_n$, for $n\in\N$. Then for $n\geq 0$,
$$
M_n = \begin{pmatrix}
r_{n-1} & r_n\\
s_{n-1} & s_n
\end{pmatrix},
$$
where $\tfrac{r_n}{s_n}$ is the $n$th convergent in the continued fraction expansion~\eqref{SomeExpansionOfX} of $x$ (note that $r_n$ and $s_n$ are not necessarily relative prime, as it is not guaranteed that the determinant of $M_n$ is $\pm 1$). So we have (and the successive convergents are in {\color{red}red}) for $n$
$$
M_n = \begin{pmatrix}
r_{n-2} & {\color{red}r_{n-1}} \\
s_{n-2} & {\color{red}s_{n-1}}
\end{pmatrix}
\begin{pmatrix}
0 & 1 \\
1 & d_n
\end{pmatrix} = \begin{pmatrix}
r_{n-1} & {\color{red}d_nr_{n-1}+r_{-2}}\\
s_{n-1} & {\color{red}d_ns_{n-1}+s_{-2}}
\end{pmatrix},
$$
yielding the usual recurrence relations for the $r_n$ and $s_n$. 

Now,
\begin{eqnarray*}
&& \begin{pmatrix} 
r_{n-2} &  {\color{red} r_{n-1}}\\
s_{n-2} &  {\color{red} s_{n-1}}
\end{pmatrix} \begin{pmatrix} 
0 & 1\\
1 & 2
\end{pmatrix} \begin{pmatrix} 
0 & 1\\
1 & 0
\end{pmatrix} \begin{pmatrix} 
0 & 1 \\
1 & d_n-2
\end{pmatrix} \\
&=& \begin{pmatrix} 
r_{n-1} & {\color{red} 2r_{n-2}+r_{n-2}} \\
s_{n-1} & {\color{red} 2s_{n-1}+s_{n-2}}
\end{pmatrix} \begin{pmatrix} 
0 & 1\\
1 & 0
\end{pmatrix} \begin{pmatrix} 
0 & 1 \\
1 & d_n-2
\end{pmatrix} \\
&=& \begin{pmatrix} 
2r_{n-2}+r_{n-2} & {\color{red} r_{n-1}} \\
2s_{n-1}+s_{n-2} & {\color{red} s_{n-1}}
\end{pmatrix} \begin{pmatrix} 
0 & 1 \\
1 & d_n-2
\end{pmatrix} \\
&=& \begin{pmatrix} 
r_{n-1} & {\color{red} 2r_{n-2}+r_{n-2} + (d_n-2)r_{n-1}}  \\
s_{n-1} & {\color{red} 2s_{n-1}+s_{n-2} + (d_n-2)s_{n-1}}
\end{pmatrix} \\
&=& \begin{pmatrix} 
r_{n-1} & d_nr_{n-1}+r_{n-2}  \\
s_{n-1} & d_ns_{n-1}+s_{n-2}
\end{pmatrix},
\end{eqnarray*}
and the last matrix is indeed $M_n$, due to the well-known recurrence relations (recall that here $c_n=1$).
\end{proof}

On $[\tfrac{1}{2},1]$, the Romik map $R$ is in fact the ``flipped'' RCF-map. Say the RCF-expansion of some $x\in (\tfrac{1}{2},1)$ is $x=[0;1,a,b,c,\dots]$, with $a,b,c\in\N$. We have two cases.

If $a=1$, we know from Proposition~\ref{prop:relationRomikGauss} that $R(x)=[0;b+1,c,\dots]$, and the reason why this is so is explained in Section~2.2.1 on pages 57 and 58 of \cite{[DHKM]}; we \emph{singularized} the regular partial quotient {\color{red}$a=1$}, to find the first RCF-convergent of $x$ (here $\tfrac{1}{1}$) deleted; see also~\eqref{singularization},
$$
\frac{1}{1 + \displaystyle \frac{1}{{\color{red}1} +  \displaystyle  \frac{1}{b +  \displaystyle  \frac{1}{c + \ddots}}}} =
\frac{1}{{\color{red}2} +  \displaystyle \frac{\color{red}-1}{{\color{red}b+1} + \displaystyle \frac{1}{c + \ddots}}};
$$

If $a>1$, we know from Proposition~\ref{prop:relationRomikGauss} that $R(x)=[0;1,a-1,b,c,\dots]$, and the reason why this is so is explained in Section~2.2.2 on pages 58 and 59 of \cite{[DHKM]}; we \emph{inserted} {\color{red}$-1/1$} before the partial quotient $a\geq 2$ to obtain an extra (mediant) convergent,
$$
\frac{1}{1 + \displaystyle \frac{1}{a +  \displaystyle  \frac{1}{b +  \displaystyle  \frac{1}{c + \ddots}}}} =
\frac{1}{2 +  \displaystyle \frac{\color{red}-1}{{\color{red}1} +  \displaystyle \frac{1}{a-1 +  \displaystyle \frac{1}{b +  \displaystyle \frac{1}{c + \ddots}}}}};
$$
see also~\eqref{insertion}.\smallskip\

As in the proof of Lemma~\ref{lemma:strangeinsertion} both cases can easily be seen using suitable matrix products. We omit these here. 

The following algorithm describes how we can convert the RCF-expansion of any $x\in\R$ to the Romik expansion of $x$. It is reminiscent of the algorithm from~\cite{[KL]}, pp.~299--300, to convert an RCF-expansion of $x$ into an ECF-expansion of $x$. Let us first introduce some notation. Let $x\in\R$ have a continued fraction expansion as in~\eqref{SomeExpansionOfX}, so $x=[d_0;c_1/d_1,c_2/d_2,\dots]$.
\begin{itemize}
\item[($i$)] If $c_n=+1$, $d_n>3$, then  
$$
\pi_n([d_0;c_1/d_1,c_2/d_2,\dots,c_{n-1}/d_{n-1},1/d_n,\dots]) {\phantom{XXXXXXX}}
$$
$$
= [d_0;c_1/d_1,c_2/d_2,\dots,c_{n-1}/d_{n-1},1/2,1/0,1/(d_n-2),\dots],
$$
i.e., we have a ``strange insertion'' ${\displaystyle \frac{1}{2+\displaystyle \frac{1}{0\,\, + }}}$ before $d_n$; see also~\eqref{strangeinsertion}.

\item[ ]

\item[($ii$)] If $d_n=1$, $c_{n+1}=+1$, $d_{n+1}=1$, $c_{n+2}=+1$, then  
$$
\sigma_n([d_0;c_1/d_1,c_2/d_2,\dots,c_n/1,1/1, 1/d_{n+2},\dots]) {\phantom{XXXXXXX}}
$$
$$
= [d_0;c_1/d_1,c_2/d_2,\dots,c_n/2,-1/(d_{n+2}+1),\dots],
$$
i.e., we have singularized $d_{n+1}=1$.

\item[ ]

\item[($iii$)] If $d_n=1$, $c_{n+1}=+1$, $d_{n+1}>1$, then  
$$
\iota_n([d_0;c_1/d_1,c_2/d_2,\dots,c_n/1,1/d_{n+1},\dots]) {\phantom{XXXXXXXXX}}
$$
$$
= [d_0;c_1/d_1,c_2/d_2,\dots,c_n/2,-1/1,1/(d_{n+1}-1),\dots],
$$
i.e., we inserted $-1/1$ before $d_{n+1}$.

\item[ ]
\end{itemize}

\begin{algorithm}\label{convertionRCFtoRomikExpansion}{\rm
Let $x\in\R$ have RCF-expansion $x=[a_0;a_1,a_2,\dots]$. In order to convert the RCF-expansion of $x$ we have two cases:\medskip\

\textbf{(I)} Let $n\geq 1$ the first index in the RCF-expansion of $x$ for which $a_n\neq 2$. If no such index exists, then $x=a_0-1+\sqrt{2}$, and we are done. I.e., $n=\inf \{ n\in\N\, |\,a_n=2\}$, where $n=\infty$ when $\{ n\in\N\, |\,a_n=2\}=\emptyset$. we consider three cases.

\begin{itemize}\item[($i$)] The case $a_n\geq 3$. In this case we have two sub-cases. If $a_n$ is \emph{even}, apply the ``strange insertion'' $\tfrac{a_n-2}{2}$-times ``before'' $a_n$. I.e., replace $x=[a_0;a_1,\dots,a_{n-1},a_n, a_{n+1},\dots]$ by
$$
\pi_{n+\tfrac{a_n-2}{2}}\left( \cdots (\pi_{n+2}(\pi_n([a_0;a_1,\dots,a_{n-1},a_n, a_{n+1},\dots]))) \cdots \right),
$$
which is equal to
\begin{equation}\label{conversionRCFtoRomik1(i)even}
x=[a_0;1/a_1,\dots,1/a_{n-1},(1/2,1/0)^{\tfrac{a_n-2}{2}},1/2,1/a_{n+1},\dots].
\end{equation}
Denote the continued fraction in~\eqref{conversionRCFtoRomik1(i)even} by $x=[d_0;c_1/d_1,c_2/d_2,\dots]$ (so $d_0=a_0$, $c_1=1=c_2=\cdots =c_{n}=c_i$, for all $i\geq n+\tfrac{a_n-2}{2}+1$, and all other $cj$ are $-1$).\medskip\

If $a_n$ is \emph{odd}, apply the ``strange insertion'' $\tfrac{a_n-1}{2}$-times ``before'' $a_n$. I.e., replace $x=[a_0;a_1,\dots,a_{n-1},a_n, a_{n+1},\dots]$ by
$$
\pi_{n+\tfrac{a_n-1}{2}}\left( \cdots (\pi_{n+2}(\pi_n([a_0;a_1,\dots,a_{n-1},a_n, a_{n+1},\dots]))) \cdots \right),
$$
which is equal to
\begin{equation}\label{conversionRCFtoRomik1(i)odd}
x=[a_0;1/a_1,\dots,1/a_{n-1},(1/2,1/0)^{\tfrac{a_n-2}{2}},1/1,1/a_{n+1},\dots].
\end{equation}
Again, denote the continued fraction in~\eqref{conversionRCFtoRomik1(i)odd} by $x=[d_0;c_1/d_1,c_2/d_2,\dots]$.

\item[ ]

\item[($ii$)] The case $a_n=1=a_{n+1}$. In this case, singularize $a_{n+1}=1$. I.e., replace $x=[a_0;a_1,\dots,a_{n-1},1,1,a_{n+2},a_{n+3},\dots]$ by
$$
\sigma_n([a_0;a_1,\dots,a_{n-1},1,1,a_{n+2},a_{n+3},\dots]),
$$
which equals
\begin{equation}\label{conversionRCFtoRomik1(ii)}
x=[a_0;1/a_1,\dots,1/a_{n-1},1/2,-1/(a_{n+2}+1),1/a_{n+3},\dots].
\end{equation}
Denote the continued fraction in~\eqref{conversionRCFtoRomik1(ii)} by $x=[d_0;c_1/d_1,c_2/d_2,\dots]$.

\item[ ]

\item[($iii$)] The case $a_n=1$, $a_{n+1}\geq 2$. In this case, start to apply insertions $a_n-1$-times ``after'' $a_n=1$ and ``before'' $a_{n+1}$. I.e., replace the regular continued fraction expansion $[a_0;a_1,\dots,a_{n-1},1, a_{n+1},a_{n+2},\dots]$ of $x$ by
$$
\iota_{n+a_{n+1}-1}(\cdots (\iota_{n+1}(\iota_n([a_0;a_1,\dots,a_{n-1},1,a_{n+1},a_{n+2},\dots])))\cdots ),
$$
which equals
\begin{equation}\label{conversionRCFtoRomik1(iii)}
x=[a_0;1/a_1,\dots,1/a_{n-1},1/2,(-1/2)^{a_{n+1}-2}, -1/1,1/a_{n+2},\dots].
\end{equation}
Denote the continued fraction in~\eqref{conversionRCFtoRomik1(ii)} by $x=[d_0;c_1/d_1,c_2/d_2,\dots]$.

\item[ ]    
\end{itemize}

\textbf{(II)} Let $n\geq 1$ the first index in the expansion of $x$ we found in step \textbf{(I)} of this algorithm, for which $d_n\neq 2$. If such an index $n$ does not exist, that is, if $d_n\neq 2$ for all $n$, then we are done and $[d_0;c_1/d_1,c_2/d_2,\dots]$ is the Romik-expansion of $x$. Otherwise, repeat step \textbf{(I)}.}\hfill $\triangledown$
\end{algorithm}

\subsection{A metric result}\label{MetricResult}
Note that in the conversion from the RCF-expansion of any $x\in\R$ to the Romik-expansion of $x$, some RCF-convergents of $x$ are ``kept'' as Romik-convergents, while other RCF-convergents of $x$ are skipped as Romik-convergents. From Algorithm~\ref{convertionRCFtoRomikExpansion}, it is clear that ``nothing happens'' whenever $R^n(x)\in [0,\tfrac{1}{3}]\cup [\tfrac{2}{3},1]$, that an RCF-convergent of $x$ is kept whenever $R^n(x)\in (\tfrac{1}{3},\tfrac{1}{2}]$, and that an RCF-convergent of $x$ is deleted  whenever $R^n(x)\in (\tfrac{1}{2},\tfrac{2}{3})$. In order to see -- for almost all $x\in [0,1]$ with respect to Lebesgue measure $\lambda$ -- what the fraction of RCF-convergents of $x$ which are also Romik-convergents among all RCF-convergents of $x$, we use \emph{Hopf's Ratio Ergodic Theorem}; see Theorem~11.5.2 on p.~219 of~\cite{[DK]}, or~\cite{[Z]}. Setting $f=1_{(\tfrac{1}{2},\tfrac{2}{3})}$ (i.e., $f$ is the indicator function of the set $(\tfrac{1}{2},\tfrac{2}{3})$), and $g=1_{(\tfrac{1}{3},\tfrac{2}{3})}$, then $g\geq 0$ and $\int_0^1 g\, {\rm d}\mu >0$, where $\mu$ is the $\sigma$-finite, infinite $R$-invariant measure with density given in~\eqref{density}. Due to Theorem~11.5.2 from~\cite{[DK]},
$$
\lim_{n\to\infty} \frac{\sum_{i=0}^{n-1}f(R^i(x))}{\sum_{i=0}^{n-1}g(R^i(x))} = \frac{\int_0^1 f\, {\rm d}\mu}{\int_0^1 g\, {\rm d}\mu},\quad \mu-\text{a.e.}
$$
As 
$$
\int_0^1 f\, {\rm d}\mu = \int_{\tfrac{1}{2}}^{\tfrac{2}{3}} \frac{1}{x(1-x)}\, {\rm d}x = \log 2,
$$
and 
$$
\int_0^1 g\, {\rm d}\mu = \int_{\tfrac{1}{3}}^{\tfrac{2}{3}} \frac{1}{x(1-x)}\, {\rm d}x = 2\log 2,
$$
we have obtained the following metrical result.

\begin{theorem}\label{thm:densityRCFconvergents}
For almost all $x\in [0,1]$, the asymptotic ratio of deleted RCF-convergents of $x$ among all RCF-convergents of $x$ is
$$
\frac{\displaystyle \int_{\tfrac{1}{2}}^{\tfrac{2}{3}} \frac{1}{x(1-x)}\, {\rm d}x}{\displaystyle \int_{\tfrac{1}{3}}^{\tfrac{2}{3}} \frac{1}{x(1-x)}\, {\rm d}x} = \frac{1}{2}.
$$
So for almost all $x$, asymptotically half of all RCF-convergents are skipped.
\end{theorem}

\end{document}